\documentclass[10pt,english]{extarticle}

\usepackage{caption}
\usepackage{enumerate}
\usepackage{t1enc}
\usepackage{babel}
\usepackage{amsthm}
\usepackage{amsmath}
\usepackage{amsfonts}
\usepackage{mathrsfs}
\usepackage[cp1250]{inputenc}
\usepackage{amsthm}
\usepackage{amsmath}
\usepackage{amssymb}
\usepackage{amsfonts}
\usepackage{indentfirst}
\usepackage{textcomp}
\usepackage{mathabx}
\usepackage{bbm}
\usepackage{tikz}
\usepackage{xcolor}
\usepackage{extsizes}
\usepackage[textwidth=18.5cm,textheight=23cm,left=3.5cm,voffset=-1cm,,footskip=1.5cm,headheight=0cm]{geometry}
\usepackage{fancyhdr}
\usepackage{yfonts}
\usepackage{float}
\usepackage{pgf}
\usepackage{pgfplots}
\usepackage{amsfonts,amsmath,amssymb,graphicx,array,url}
\usepackage{polski}
\usepackage[cp1250]{inputenc}
\pgfplotsset{compat=1.12}

\makeatletter
\def\namedlabel#1#2{\begingroup
    #2%
    \def\@currentlabel{#2}%
    \phantomsection\label{#1}\endgroup
}
\makeatother

\begin{document}

\title{On general financial markets with concave transactions costs}

\author{Agnieszka Rygiel \footnote{Krakow University of Economics, Department of Mathematics, Email: 
rygiela@uek.krakow.pl, the author was financed by the project no. 089/EIM/2024/POT of Krakow University of Economics.} \and
Lukasz Stettner  \footnote{Institute of Mathematics
 Polish Academy of Sciences,
  Sniadeckich 8, 00-656 Warsaw, Email: stettner@impan.pl, the author acknowledge research support by
Polish National Science Centre grant no. 2020/37/B/ST1/00463. Part of
this work was completed with the help of  University of Warsaw grant
IDUB - POB3-D110-003/2022 and Simons Semester ``Stochastic Modeling
and Control'' at IMPAN in May-June 2023.}
}
\maketitle

\begin{abstract}In the paper we study markets with concave transaction costs which depend in a concave way on 
the volume of
transaction. This is typical situation in the case of small investors, which commonly appears in currency 
and real estate markets.  
Sufficient conditions for absence of arbitrage are formulated. New notion of asymptotic arbitrage is 
introduced and used to study the 
above mentioned markets.
\end{abstract}



\theoremstyle{plain}
\setlength{\parskip}{12pt plus0pt minus12pt}
\newtheorem{Theorem}{Theorem}[section]
\newtheorem{Lemma}[Theorem]{Lemma}
\newtheorem{Corollary}[Theorem]{Corollary}
\newtheorem{Notation}[Theorem]{Notation}
\newtheorem{Proposition}[Theorem]{Proposition}
\newtheorem{Fact}[Theorem]{Fact}
\newtheorem{Definition}[Theorem]{Definition}

\theoremstyle{definition}
\newtheorem{Remark}[Theorem]{Remark}
\newtheorem{Example}[Theorem]{Example}

\theoremstyle{remark}
\newtheorem{Convention}{Convention}[section]
\newtheorem{Assumption}{Assumption}[section]
\numberwithin{equation}{section}
\def\ve{\varepsilon}

\section{Introduction}
On a given probability space $(\Omega, {\cal F}, ({\cal F}_t), \mathbb{P})$ we consider a discrete time market 
with concave
transaction costs. We have bank and asset accounts. The prices of asset depend on a number of assets we buy or 
sell. We have a bid
curve $0\leq m\to
\underline{S}_t(m)$ for which we sell $m$ assets and an ask curve $0\leq l\to \overline{S}_t(l)$ for which we 
buy $l$ assets, which
are adapted to ${\cal F}_t$. Since we also allow fixed transaction costs: we can sell $m$ assets at time $t$ 
obtaining
$b_t(m)=b_0+m\underline{S}_t(m)$, or buy $l$ assets spending $a_t(l)=a_0+l\overline{S}_t(l)$ at time $t$, where 
$b_0\leq 0$ while
$a_0\geq 0$.
In the paper we shall assume that
\begin{enumerate}
\item[(a1)]the mapping $0\leq m\to \underline{S}_t(m)$ is increasing and continuous with limit 
    $\underline{S}_t(\infty)$ at 
    $\infty$, for
sufficiently large $m$ is differentiable and $\lim_{m\to \infty} m \underline{S}'_t(m)=0$,
\item[(a2)]
the mapping $0\leq m\to b_t(m)$ is increasing convex,
\item[(a3)]
the mapping $0\leq l\to \overline{S}_t(l)$ is decreasing and continuous with limit $\overline{S}_t(\infty)$ 
at $\infty$, for
sufficiently large $l$ is differentiable and $\lim_{l\to \infty} l \overline{S}'_t(l)=0$,
\item[(a4)]
the mapping $0\leq l\to a_t(l)$ is increasing concave,
\item[(a5)]
furthermore $\underline{S}_t(\infty)<\overline{S}_t(\infty)$.
\end{enumerate}

In what follows all equalities and inequalities will be considered $\mathbb{P}$ almost everywhere. 
\begin{Remark}
Although we generally interested in strictly concave (concave)  mappings $0\leq l\to a_t(l)$ ($0\leq m\to 
b_t(m)$) we admit the case
when they are simply concave (convex) to cover the case with fixed plus proportional transaction costs where 
$a_t(l)=a_0+l\overline{S}_t$ and
$b_t(m)=b_0+m\underline{S}_t$, where $\overline{S}_t$ and $\underline{S}_t$ do not depend on $l$ or $m$ respectively.
Since the mapping  $0\leq l\to a_t(l)$ ($0\leq m\to b_t(m)$) is concave (convex) it is continuous and therefore 
we have continuity of
$0\leq l\to \overline{S}_t(l)$
($0\leq m\to \underline{S}_t(m)$) for $l>0$ ($m>0$).
\end{Remark}

We have the following consequences of our assumptions on bid and ask curves

\begin{Lemma}\label{lemmaineq}
We have
\begin{equation}\label{ineq1}
a_t(l_1)-a_t(l_1-l_2)\leq a_t(l_2) \ \ for \ \ 0<l_2<l_1
\end{equation}
\begin{equation}\label{ineq2}
b_t(m_1)-b_t(m_1-m_2)\geq b_t(m_2) \ \ for \ \ 0<m_2<m_1
\end{equation}
\begin{equation}\label{ineq3}
\overline{S}_t(0)\geq  {a_t(l_1)-a_t(l_1-l_2)\over l_2}\geq \overline{S}_t(\infty) \ \ for \ \ 0<l_2<l_1
\end{equation}
\begin{equation}\label{ineq4}
\underline{S}_t(0)\leq {b_t(m_1)-b_t(m_1-m_2)\over m_2}\leq\underline{S}_t(\infty)  \ \ for \ \ 0<m_2<m_1
\end{equation}
 \begin{equation}\label{ineq5}
a_t(l_1-l_2) < a_t(l_1)-b_t(l_2) \ \ for \ \ 0<l_2<l_1
\end{equation}
\begin{equation}\label{ineq6}
b_t(m_1-m_2)>b_t(m_1)-a_t(m_2) \ \ for \ \ 0<m_2<m_1
\end{equation}
with strict inequalities in \eqref{ineq1} and \eqref{ineq3} or  \eqref{ineq2} and \eqref{ineq4} when the 
mapping $0\leq l\to a_t(l)$
or $0\leq m\to b_t(m)$ is strictly concave or strictly convex respectively.
\end{Lemma}
\begin{proof} Inequalities \eqref{ineq1} and \eqref{ineq2}  follow directly from  concavity or  convexity of 
the mappings
 $0\leq l\to a_t(l)$ and $0\leq m\to b_t(m)$. To show \eqref{ineq3} we use concavity of $0\leq l\to a_t(l)$ and 
 the fact that
 $\lim_{l\to \infty} l \overline{S}'_t(l)=0$. Similarly to show \eqref{ineq4} we use convexity of $0\leq m\to 
 \underline{S}_t(m)$ and
 $\lim_{m\to \infty} m \underline{S}'_t(m)=0$. \eqref{ineq5} follows directly from \eqref{ineq3} and assumption 
 that
 $\underline{S}_t(\infty)<\overline{S}_t(\infty)$ and \eqref{ineq6} can be shown from \eqref{ineq4}.
 \end{proof}
Inequalities \eqref{ineq1}, \eqref{ineq2}, \eqref{ineq5}, \eqref{ineq6} are consistent with our intuition. 
Namely by \eqref{ineq1}  it
is cheaper to buy $l_1$ assets than make two immediate transactions buying first $l_1-l_2$ and then $l_2$ 
assets. Similarly we have
with selling assets (see \eqref{ineq2}). Accordingly to \eqref{ineq5} it is cheaper to buy $l_1-l_2$ assets 
than buy $l_1$ and sell
$l_2$ assets. Similarly we have with selling assets in \eqref{ineq6}. We have the following example of bid and ask 
curves

\begin{Example}\label{ex1}
Let for $m\geq 0$ and $l\geq 0$ with $0<\alpha<1$
\begin{equation}\label{exbid}
\underline{S}_t(m)=\underline{S}_t(\infty)-{\underline{S}_t(\infty)-\underline{S}_t(0) \over (m+1)^\alpha}
\end{equation}
\begin{equation}\label{exask}
\overline{S}_t(l)=\overline{S}_t(\infty)+{\overline{S}_t(0)-\overline{S}_t(\infty) \over (l+1)^\alpha}
\end{equation}
where $\overline{S}_t(\infty)\geq \underline{S}_t(\infty)$.
In particular when
$\underline{S}_t(\infty)=(1+p)\underline{S}_t(0)$ and $\overline{S}_t(\infty)=(1-q)\overline{S}_t(0)$ with 
$p\geq 0$, $0\leq q<1$
we have
\begin{equation}\label{exbidp}
\underline{S}_t(m)=\underline{S}_t(0){(1+p)(m+1)^\alpha-p \over (m+1)^\alpha}
\end{equation}
and
\begin{equation}\label{exaskp}
\overline{S}_t(l)=\overline{S}_t(0){(1-q)(l+1)^\alpha + q \over (l+1)^\alpha}.
\end{equation}
We can easily check that these curves satisfy all assumptions imposed above.
\end{Example}
In what follows it will be convenient to write
\begin{equation}\label{notat}
a_t(l)=a(l,\overline{S}_t):=a_0+l \overline{S}_t(l)   \ \ and \ \ b_t(m)=b(m,\underline{S}_t):=b_0+m 
\underline{S}_t(m)
\end{equation}
to point out dependence of $a_t$ ($b_t$) on the ask curve $l\to a_t(l)$ (bid curve $m\to b_t(m)$). We shall 
denote by $(x_t,y_t)$
amount of money $x_t$ on our bank account and number of assets $y_t$ we have in our portfolio at time $t$ 
before possible
transactions. Given bank-stock position $(x,y)$ at time $t$ after selling $m$ assets and buying $l$ assets our 
position is $(x+m
\underline{S}_t - l \overline{S}_t, y-m+l)$. We can also liquidate our position $(x,y)$ at time $t$ 
transferring it to bank account
for which we use liquidation function $L_t$:
\begin{eqnarray}\label{liquid}
&&L_t(x,y)=\mathbbm{1}_{y\geq 0} \left(x+(b_0+y \underline{S}_t(y))^+\right) + \mathbbm{1}_{y<0} \left(x- a_0+ 
y
\overline{S}_t(-y)\right) = \nonumber \\
&&\mathbbm{1}_{y\geq 0} 
\left(x+(b(y,\underline{S}_t))^+\right)+\mathbbm{1}_{y<0}\left(x-a(-y,\overline{S}_t)\right)=x+L_t(0,y).
\end{eqnarray}
Positive part of $b_0+y \underline{S}_t$ means that we do not sell assets when possible income is below fixed 
transaction cost $-b_0$.
We can define the set $G_t$ of solvent positions at time $t$:
\begin{equation}\label{solv}
G_t:=\left\{(x,y): L_t(x,y)\geq 0\right\}.
\end{equation}
We easily see that $G_t$ is a closed subset of  $\mathbb{R}^{2}$. It contains $\mathbb{R}_+^{2}$ and is bounded 
by two concave curves:
$y=-a_t^{-1}(x)$ for $x\geq a_0$, and $y=b_t^{-1}(-x)$ for $x\leq 0$. Consequently $G_t$ is not a convex set.
We can equivalently write that
\begin{equation}\label{eqdefL}
L_t(x,y)=\max\left\{\alpha: (x-\alpha,y)\in G_t \right\},
\end{equation}
so that there is a $1-1$ relation between $L_t$ and $G_t$. We shall write sometimes $L^G$ to point out 
dependence of $L_t$ on $G_t$.
Directly from \eqref{eqdefL} we have
\begin{Corollary}\label{zeronpart}
 $L_t(x,y)=0$ for any $(x,y)\in \partial G_t$ if and only if $a_0=0$. When $a_0>0$ we have $L_t(x,0)>0$ for 
 $x\in (0,a_0]$ and 
 $L_t(x,y)=0$ for $(x,y)\in \partial G_t\setminus \left\{(0,a_0]\times \left\{0\right\}\right\}$.
\end{Corollary}
Therefore in what follows we shall understand $\partial G_t$ as the set $\left\{(x,y): L_t(x,y)=0\right\}$ and 
interior $G_t^0$ of 
$G_t$ as $G_t^0:=G_t \setminus \partial G_t=\left\{(x,y): L_t(x,y)>0 \right\}$. In what follows market with concave transaction costs will be called $(G_t)$ market to point out its dependence on the solvency sets $(G_t)$.   
From \eqref{liquid} we obtain
\begin{Lemma}
We have
\begin{equation}\label{liqin}
L_t((x,y)+(\bar{x},\bar{y}))\geq L_t(x,y)+L_t(\bar{x},\bar{y}).
\end{equation}
\end{Lemma}
\begin{proof} When $y, \bar{y}\geq 0$ \eqref{liqin} follows from $(b(y+\bar{y}, \underline{S}_t))^+\geq 
(b(y,\underline{S}_t))^+
+ (b(\bar{y}, \underline{S}_t))^+$. The case when $y, \bar{y}< 0$ is also immediate. Consider the case when 
$y\geq 0$ and $\bar{y}<0$.
When $y+\bar{y}\geq 0$ we have $L_t((x,y)+(\bar{x},\bar{y}))=x+\bar{x}+(b(y+\bar{y},\underline{S}_t))^+$ and 
using \eqref{ineq6} we
obtain $b(y,\underline{S}_t)-a(-\bar{y},\overline{S}_t)\leq b(y+\bar{y},\underline{S}_t)$. In the case when 
$y+\bar{y}<0$ we have
$L_t((x,y)+(\bar{x},\bar{y}))=x+\bar{x}-a(-(y+\bar{y}),\overline{S}_t)$ and using \eqref{ineq5} we obtain
$b(y,\underline{S}_t)-a(-\bar{y}, \overline{S}_t)\leq-a(-(y+\bar{y}),\overline{S}_t)$.
\end{proof}
Immediately from \eqref{liqin} we obtain
\begin{Corollary} We have
\begin{equation}\label{gcont}
G_t+G_t \subset G_t.
\end{equation}
\end{Corollary}
We also have that
\begin{Lemma}
$(x,y)\to L_t(x,y)$ is continuous for $y\neq 0$. When $a_0=0$ it is continuous everywhere.
\end{Lemma}
and
\begin{Lemma}
For $\lambda\geq 1$ we have for $l,m\geq 0$
\begin{equation}
a(\lambda l, \overline{S}_t)\leq \lambda a(l, \overline{S}_t),  \ \ b(\lambda m, \underline{S}_t)\geq \lambda 
b(m, \underline{S}_t).
\end{equation}
\end{Lemma}
\begin{proof}
By concavity we have
$ a(l, \overline{S}_t)\geq {1\over \lambda} a(\lambda l, \overline{S}_t) + (1-{1\over \lambda})a(0, 
\overline{S}_t)$.
Therefore
$ a(\lambda l, \overline{S}_t)\leq \lambda a(l, \overline{S}_t) - (\lambda - 1)a_0 \leq \lambda a(l, 
\overline{S}_t)$.
Similarly by convexity
$b(m, \underline{S}_t)\leq {1\over \lambda} b(\lambda m, \underline{S}_t) + (1-{1\over \lambda})b(0, 
\underline{S}_t)$
and taking into account that $b_0\leq 0$,
$b(\lambda m, \underline{S}_t)\geq \lambda b(m, \underline{S}_t) - (\lambda - 1) b_0 \geq\lambda b(m, 
\underline{S}_t)$,
which completes the proof.
\end{proof}
Using the last Lemma we obtain
\begin{Corollary}\label{cor1.8}
For $\lambda \geq 1$ we have
\begin{equation}
L_t(\lambda(x,y))\geq \lambda L_t(x,y) \ \  and \ \ \lambda G_t\subset G_t.
\end{equation}
\end{Corollary}
Markets with proportional transaction costs have been studied intensively in a number of papers (see e.g. \cite{KS}, \cite{Zh} and  references therein).  The case 
of illiquid prices depending on the volume of transaction has appeared first in the papers \cite{CJP} and \cite{UR} and the idea was continued in the papers \cite{P1}-\cite{P3}. The
form of transaction costs depends on the impact of the investor on
the market and therefore on his ability to obtain lower costs. In the case of small investor bid 
price increases with the 
volume of transaction while ask price decreases. Therefore we have concave transaction costs. Such situation is frequent in particular in the case of currency or real estate markets. In the case of large investors, large sales diminish bid price and large purchases increase ask price. This leads to convex transaction costs. In the paper we study concave transaction costs. We also consider so called broker's fees composed of fixed plus concave transaction costs. The case of convex transaction costs was studied in the papers \cite{GM}, \cite{P1}-\cite{P3} and \cite{MSXZ}. Fixed plus proportional transaction costs were considered in \cite{HH}. Convex transaction costs appeared in \cite{DTP}. The case of general transaction costs was considered in the papers \cite{EL}, \cite{LT1} and \cite{LT2}. In the last paper there are some gaps which are now corrected in section 3 in which we follows methodology of section 6 of \cite{LT2} .  In the paper, the first asset is the discounted value of a
non risky asset, i.e. $S_0 = 1$ for the cash financial position $x$ while only one
risky asset defines the risky position $y$ of any investment $(x, y)$.
 Sections 2 and 3 may be easily extended to the case of several risky assets.
In section 4 and 5 we introduce so called asymptotic arbitrage, which is useful to study markets with concave transaction costs. Namely we show that arbitrage on the market with proportional transaction costs implies an asymptotic arbitrage on the market with concave transaction costs.         

\section{Absence of arbitrage}

We can easily notice that $-G_t$ forms the set of positions which we can achieve starting at time $t$ from 
position $(0,0)$.
A portfolio is by definition a stochastic process
which is adapted to the filtration $({\cal F}_t)_{t\leq T}$ where ${\cal F}_t$  describes the information
available at time $t$ or before. That means that at time $t$ any trade is chosen
as a function of the information ${\cal F}_t$. 
 We shall denote by
$L^0(-G_t,{\cal F}_t)$ the set of all $-G_t$ valued ${\cal F}_t$ measurable random variables. Given initial 
position $(x,y)$ we can
choose initial portfolio $V_0=(x,y)+(\xi_0,\zeta_0)$ where $(\xi_0,\zeta_0)\in L^0(-G_0,{\cal F}_0)$, and by 
induction
$V_{t+1}=V_t+(\xi_{t+1},\zeta_{t+1})$, where $(\xi_{t+1},\zeta_{t+1})\in L^0(-G_{t+1},{\cal F}_{t+1})$ at time 
$t+1$, for
$t=0,1,\ldots$.
Such portfolio is self financing since it does not require external transfer of capital.
Let
\begin{equation}\label{r01}
R^0_T=\sum_{t=0}^T L^0(-G_{t},{\cal F}_{t}).
\end{equation}
It is the set of all possible final values of portfolio at time $T$ when we start with initial position $(0,0)$ and use self financing strategy.
Let
\begin{equation}\label{r02}
LV^0_T=\left\{L_T(V_T): V_T\in R^0_T\right\}.
\end{equation}
Market with concave transaction costs satisfies {\it absence of arbitrage condition} (NA), when $LV^0_T \cap L^0(\mathbb{R}_+,{\cal F}_T)=\left\{0\right\}$, where $L^0(\mathbb{R}_+,{\cal F}_T)$ is the set of $\mathbb{R}_+$ valued $({\cal F}_T)$ measurable random variables.
We can also say that we have an {\it arbitrage} (A) on the market with concave transaction costs whenever there is $\zeta_T\in R^0_T$ such that $L_T(\zeta_T)\geq 0$, $\mathbb{P}$ a.e. and $\mathbb{P}\left\{L_T(\zeta_T)>0\right\}>0$ where $\left\{0\right\}$ stands for a random variable that is  $\mathbb{P}$ a.e. equal to $0$.
In other words $\zeta_T\in G_T$, $\mathbb{P}$ a.e. and 
$\mathbb{P}\left\{\zeta_T\in G_T^0\right\}>0$.
 In the case when there is $V_T^n=\sum_{t=0}^T \xi_t^n$ such that
 $\lim_{n\to \infty}\mathbb{P}\left\{\forall_t  \ \xi_t^n\in -G_t\right\}=1$  and $\liminf_{n\to \infty}\mathbb{P}\left\{L_T(V_T^n)\geq 0\right\}=1$ and also $\liminf_{n\to \infty}\mathbb{P}\left\{L_T(V_T^n) >0\right\}>0$ we say that we have an {\it asymptotic arbitrage} (AA). It is clear that (A) implies (AA). 

Let
\begin{equation}
A^0_T=\left\{\sum_{i=1}^n \lambda_i X_i,  \ where \ \lambda_i\geq 0 \ and \ X_i\in LV^0_T\right\}.
\end{equation}
Clearly $A^0_T$ is a convex cone in $L^0(\mathbb{R},{\cal F}_T)$. Following  Lemma 3.2 of \cite{LT2} we have 
that
\begin{equation}
A^0_T=\left\{\lambda X, \ where \ \lambda \in [0,1] \ and \ X\in LV^0_T\right\}.
\end{equation}
We have
\begin{Lemma}\label{NAequi} Conditions $LV^0_T \cap L^0(\mathbb{R}_+,{\cal F}_T)=\left\{0\right\}$ and $A^0_T \cap L^0(\mathbb{R}_+,{\cal F}_T)=\left\{0\right\}$ are equivalent.
\end{Lemma}
\begin{proof}
Clearly $LV^0_T\subset A^0_T$, so that $A^0_T \cap L^0(\mathbb{R}_+,{\cal F}_T)=\left\{0\right\}$ implies (NA). 
Assume that there is
$X\in LV^0_T$ such that $\lambda X\geq 0$, $\mathbb{P}$ a.e. and $\mathbb{P}\left\{\lambda X>0\right\}>0$ for 
some $\lambda\in [0,1]$.
Therefore also $X\geq 0$, $\mathbb{P}$ a.e. and $\mathbb{P}\left\{X>0\right\}>0$, which means that (NA) does 
not hold.
\end{proof}
Let
\begin{equation}\label{K set}
K_t=conv G_t,
\end{equation}
where conv stays for a convex hull. One can notice that
$\bar{K}_t$ is a closed cone with boundaries $y={-x \over \underline{S}_t(\infty)}$ for $x\leq 0$ and $y={-x 
\over
\overline{S}_t(\infty)}$ for $x\geq 0$. Such cone corresponds to solvent positions on the market with 
proportional transaction costs
with bid price $\underline{S}_t(\infty)$ and ask price $\overline{S}_t(\infty)$. The following graph shows the boundaries of the sets $G_t$ and $K_t$

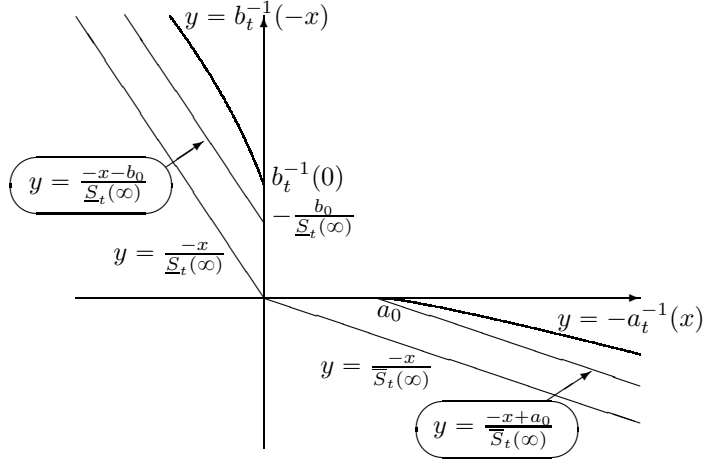
\setlength{\unitlength}{0.5 cm}
\begin{picture}(8,12)(-14,-4)
\put(-5,0){\vector(1,0){15}}
\put(3,-0.5){$a_{0}$}
\put(0,-4){\vector(0,1){11.5}}
\put(0,0){\line(-2,3){5}}
\put(0,2){\line(-2,3){3.7}}
\put(0,0){\line(3,-1){10}}
\put(3,0){\line(3,-1){7}}
\put(0.2,2)
{$-\frac{b_0}{\underline{S}_t(\infty)}$}
\put(-4,1)
{$y=\frac{-x}{\underline{S}_t(\infty)}$}
\put(-6.25,2.9)
{$y=\frac{-x-b_0}{\underline{S}_t(\infty)}$}
\put(-4.6, 3){\oval(4.3, 1.5)}
\put(-2.6,3.5){\vector(3,2){1}}
\put(1.5,-2)
{$y=\frac{-x}{\overline{S}_t(\infty)}$}
\put(4.5,-3.6)
{$y=\frac{-x+a_0}{\overline{S}_t(\infty)}$}
\put(6.2, -3.5){\oval(4.3, 1.5)}
\put(7.5,-2.75){\vector(3,2){1.2}}
\put(0.2,3)
{$b^{-1}_t(0)$}
\put(-2.1,7.3)
{$y=b^{-1}_t(-x)$}
\put(7.8,-0.75)
{$y=-a^{-1}_t(x)$}

\qbezier(0,3)(-0.8,5.2)
(-2.5,7.5)
\qbezier(3,0)(4, 0.01)
(10,-1.5)

\end{picture}

Let
\begin{equation}
{\bar{R}}^0_T=\sum_{t=0}^T L^0(-\bar{K}_{t},{\cal F}_{t}).
\end{equation}
It is the set of final positions when we use self financing strategies on the market with bid 
$\underline{S}_t(\infty)$ and ask
$\overline{S}_t(\infty)$ prices. Let
\begin{equation}
{\tilde{R}}^0_T=\sum_{t=0}^T L^0(-{K}_{t},{\cal F}_{t}).
\end{equation}

Define $\bar{L}_t(x,y):=\sup\left\{\alpha: (x-\alpha,y)\in \bar{K}_t\right\}$. Clearly we have that
$\bar{L}_t(x,y):=\sup\left\{\alpha: (x-\alpha,y)\in {K}_t\right\}$ and
\begin{equation}
\bar{L}_t(x,y)=\mathbbm{1}_{y\geq 0} \left(x+y \underline{S}_t(\infty)\right) + \mathbbm{1}_{y<0} \left(x+y
\overline{S}_t(\infty)\right)=x+\bar{L}_t(0,y).
\end{equation}
Let
\begin{equation}
\widebar{LV}^0_T=\left\{\bar{L}_T(V_T): V_T\in \bar{R}^0_T\right\}
\end{equation}
and
\begin{equation}
\widetilde{LV}^0_T=\left\{\bar{L}_T(V_T): V_T\in \tilde{R}^0_T\right\}.
\end{equation}
Clearly $\widebar{LV}^0_T$ and $\widetilde{LV}^0_T$ are cones in $L^0(\mathbb{R},{\cal F}_T)$ and  
$\widetilde{LV}^0_T\subset
\widebar{LV}^0_T$. Furthermore, $L_t(x,y)\leq \widebar{L}_t(x,y)$. {\it Absence of arbitrage} $(\widebar{NA})$ on the market 
with bid
$\underline{S}_t(\infty)$ and ask $\overline{S}_t(\infty)$ prices means that $\widebar{LV}^0_T \cap 
L^0(\mathbb{R}_+,{\cal
F}_T)=\left\{0\right\}$. An analog $(\widetilde{NA})$ of $(\widebar{NA})$ in the case of cones $K_t$ is in the form
$\widetilde{LV}^0_T \cap L^0(\mathbb{R}_+,{\cal F}_T)=\left\{0\right\}$. Similarly we denote by $(\widetilde{A})$ 
or $(\widebar{A})$ an
arbitrage on the market with solvent sets $K_t$ or $\bar{K}_t$ respectively, which in what follows will be called  $(K_t)$ or $(\bar{K}_t)$ markets. Since $\widetilde{LV}^0_T \subset 
\widebar{LV}^0_T$ we 
clearly have that $(\widebar{NA})$ implies $(\widetilde{NA})$. Furthermore each $V_T\in \bar{R}^0_T$ can be approximated by $V_T^n\in 
\tilde{R}^0_T$ such that $V_T^n\to V_T$ and $\bar{L}_T(V_T^n) \to \bar{L}_T(V_T)$, $\mathbb{P}$ a.e. as $n\to \infty$.
Therefore we have
\begin{Lemma}
(A) implies $(\widetilde{A})$ and also $(\widebar{A})$. 
\end{Lemma}
\begin{proof}
It is clear that $G_t\subset K_t$. Therefore $ R_T^0 \subset \tilde{R}_T^0$. If $\zeta_T\in R_T^0$ and $\zeta_T$ is non zero random variable and takes values in $G_T$, then we clearly have an arbitrage on the market with solvent sets $(K_t)$ and consequently also on the market with solvent sets $(\bar{K}_t)$.
\end{proof} 
\begin{Remark} Inverse implications are not necessary true. We shall prove later a result that $(\widebar{A})$ implies (AA).
\end{Remark}   
Using some ideas of the Lemma 3.4 of  \cite{LT2} we have the following equivalences
\begin{Lemma}\label{equival}  We have
\begin{enumerate}
\item[(a)] under assumption $a_0=0$ the property (NA) is equivalent to $R^0_T \cap L^0(G_T^0\cup 
    \left\{0\right\},{\cal
    F}_T)=\left\{0\right\}$; moreover if $a_0\neq 0$ we have only the implication: from (NA) it follows that  $R^0_T \cap 
    L^0(G_T^0\cup
    \left\{0\right\},{\cal F}_T)=\left\{0\right\}$,
\item[(b)] $(\widebar{NA})$ is equivalent to $\bar{R}^0_T \cap L^0(K_T,{\cal F}_T)=\left\{0\right\}$,
\item[(c)] $(\widetilde{NA})$ is equivalent to $\tilde{R}^0_T \cap L^0(K_T,{\cal F}_T)=\left\{0\right\}$.
\end{enumerate}
\end{Lemma}
\begin{proof}
If $\xi \in R^0_T \cap L^0(G_T^0\cup \left\{0\right\},{\cal F}_T)$ then $L_T(\xi)\geq 0$ and under (NA) we have $L_T(\xi)=0$ and finally $\xi=0$.
Conversely, assume that $R^0_T \cap L^0(G_T^0\cup \left\{0\right\},{\cal F}_T)=\left\{0\right\}$.  When 
$L_T(\xi)\geq 0$ for $\xi
\in R^0_T$, then $\xi\in G_T$. Clearly $\xi=\sum_{i=0}^T \xi_i$, with $\xi_i\in -G_i$. We have that either 
$\xi\in G_T^0$ or 
$\xi\in\partial G_T$. Let $B_T=\left\{\xi\in \partial G_T\right\}$. Define $\bar{\xi}_T=(\xi_T-\xi)1_{B_T} + 
\xi_T 1_{B_T^c}\in -G_T$. 
Clearly
$0=\sum_{i=0}^{T-1}\xi_i+\bar{\xi}_T$ on $B_T$.  Moreover we have that 
$\tilde{\xi}:=\sum_{i=0}^{T-1}\xi_i+\bar{\xi}_T\in R^0_T$.
Notice that $\tilde{\xi}\in G_T^0\cup \left\{0\right\}\in R^0_T$ and since $R^0_T \cap L^0(G_T^0\cup 
\left\{0\right\},{\cal
F}_T)=\left\{0\right\}$ we have that $\tilde{\xi}=0$. Then $\tilde{\xi}=\xi=0$ on $B_T^c$. Since $L_T(\xi)=0$ 
on $B_T$ (because
$a_0=0$ by Corollary \ref{zeronpart}) we finally have that  $L_T(\xi)=0$ and $(NA)$ is satisfied.
The proof of (b) and (c) follows in a similar way.
\end{proof}
\begin{Remark} In the proof above we used Corollary \ref{zeronpart}. Notice that  $\bar{L}_T(\xi)=0$ for 
$\xi\in \partial K_T$.
\end{Remark}

\begin{Corollary}
We have that $(\widetilde{NA})$ under assumption $a_0=0$ implies (NA).
\end{Corollary}
\begin{proof}
Since $R_T^0\subset \tilde{R}_T^0$ and $G_T^0\cup\left\{0\right\}\subset K_T$ whenever $\tilde{R}^0_T \cap 
L^0(K_T,{\cal
F}_T)=\left\{0\right\}$ then also $R^0_T \cap L^0(G_T^0\cup \left\{0\right\},{\cal F}_T)=\left\{0\right\}$ and 
we use Lemma
\ref{equival}.
\end{proof}

Consider the following conditions
\begin{equation}\label{tA.1} \tag{$\widebar{nA.1}$}
\bar{R}^0_T \cap L^0(\bar{K}_{t},{\cal F}_{t})=\left\{0\right\} \ for \  every  \ t=0,1,\ldots,T,
\end{equation}
and
\begin{equation}\label{tbA.1} \tag{$\widetilde{nA.1}$}
\tilde{R}^0_T \cap L^0({K}_{t},{\cal F}_{t})=\left\{0\right\} \ for \  every  \ t=0,1,\ldots,T.
\end{equation}
Following section 3.2.2. of \cite{KS} we have
\begin{Lemma} \label{lemp} (see also Lemma 3.27 of \cite{KS})
The following conditions are equivalent
\begin{enumerate}
\item[(a)] \eqref{tA.1}, i.e. $\bar{R}^0_T \cap L^0(\bar{K}_{t},{\cal F}_{t})=\left\{0\right\}$  for every $t=0,1,\ldots,T$,
\item[(b)] when $\sum_{i=0}^T \xi_i=0$ with $\xi_i \in L^0(\bar{K}_{i},{\cal F}_{i})$ then all $\xi_i=0$,
\item[(c)] for each $t\leq T$ we have $\bar{R}^0_t \cap L^0(\bar{K}_{t},{\cal F}_{t})=\left\{0\right\}$.
\end{enumerate}
\end{Lemma}
\begin{proof} When $\sum_{i=0}^T \xi_i=0$ with $\xi_i \in L^0(\bar{K}_{i},{\cal F}_{i})$ and some $\xi_t\neq 
0$, then $\bar{R}^0_T\ni
-\sum_{{i=0}, i\neq t}^T \xi_i=\xi_t$ and under \eqref{tA.1} we have  $\xi_t=0$. Conversely under (b)
when  $-\sum_{i=0}^T \xi_i=\eta_t\in L^0(\bar{K}_{t},{\cal F}_{t})$ then $\bar{R}^0_T\ni 
-(\xi_t+\eta_t)-\sum_{{i=0}, i\neq t}^T
\xi_i=0$ and each $\xi_i=0$ for $i\neq t$ and $\xi_t+\eta_t=0$. Since $\xi_t\in \bar{K}_t$ and $\eta_t\in 
\bar{K}_t$ the last can
happen only when $\xi_t=0$ and $\eta_t=0$. Therefore we have (a). Now under (a) also (c) holds since 
$\bar{R}^0_t\subset \bar{R}^0_T$.
To complete the proof it remains to show that (c) implies (b). Assume that $\sum_{i=0}^T \xi_i=0$ with $\xi_i 
\in
L^0(\bar{K}_{i},{\cal F}_{i})$. Let $\bar{s}:=\max\left\{i: \xi_i\neq 0\right\}$ and $\bar{s}=\infty$ when 
$\xi_i=0$ for each $i\leq
T$. We have to consider the case $\bar{s}<\infty$. Then $-\sum_{i=0}^{\bar{s}-1}\xi_i=\xi_{\bar{s}}$ and since
$-\sum_{i=0}^{\bar{s}-1}\xi_i\in \bar{R}^0_{\bar{s}}$ by (c) we have that $\xi_{\bar{s}}=0$ which contradicts 
definition of $\bar{s}$.
Consequently $\bar{s}=\infty$ and (b) is satisfied.
\end{proof}
\begin{Lemma}\label{timpl}
We have that \eqref{tA.1} implies  $(\widebar{NA})$.
\end{Lemma}
\begin{proof}
Assume that $(\widebar{NA})$ is not satisfied. Then there is $\xi_i\in L^0(\bar{K}_i,{\cal F}_i)$ such that 
$\bar{L}_T(-\sum_{i=0}^T
\xi_i)\geq 0$, $\mathbb{P}$ a.s., and $\mathbb{P}\left\{\bar{L}_T(-\sum_{i=0}^T \xi_i)>0 \right\}>0$. Therefore 
$\bar{R}^0_T \ni
-\sum_{i=0}^T \xi_i \in \bar{K}_T$ and $\mathbb{P}\left\{(-\sum_{i=0}^T \xi_i)\in K_T\setminus \left\{0\right\} 
\right\}>0$. This
contradicts \eqref{tA.1}.
\end{proof}
Furthermore by similar considerations as above we have
\begin{Lemma} \label{blemp}
The following conditions are equivalent
\begin{enumerate}
\item[(a)] \eqref{tbA.1}, i.e. $\tilde{R}^0_T \cap L^0({K}_{t},{\cal F}_{t})=\left\{0\right\}$  for every  $t=0,1,\ldots,T$,
\item[(b)] when $\sum_{i=0}^T \xi_i=0$ with $\xi_i \in L^0({K}_{i},{\cal F}_{i})$ then all $\xi_i=0$,
\item[(c)] for each $t\leq T$ we have $\tilde{R}^0_t \cap L^0({K}_{t},{\cal F}_{t})=\left\{0\right\}$.
\end{enumerate}
\end{Lemma}
Moreover
\begin{Lemma}\label{btimpl}
We have that \eqref{tbA.1} implies  $(\widetilde{NA})$.
\end{Lemma}

Consider now the following version of \eqref{tA.1} for the sets $G_t$
\begin{equation}\label{A.1}\tag{$nA.1$}
{R}^0_T \cap L^0({G}_{t},{\cal F}_{t})=\left\{0\right\} \ for \  every  \ t=0,1,\ldots,T.
\end{equation}
We have
\begin{Corollary} \label{glemp}
The following conditions are equivalent
\begin{enumerate}
\item[(a)] \eqref{A.1}, i.e. ${R}^0_T \cap L^0({G}_{t},{\cal F}_{t})=\left\{0\right\}$ for every $t=0,1,\ldots,T$,
\item[(b)] when $\sum_{i=0}^T \xi_i=0$ with $\xi_i \in L^0({G}_{i},{\cal F}_{i})$ then all $\xi_i=0$,
\item[(c)] for each $t\leq T$ we have ${R}^0_t \cap L^0({G}_{t},{\cal F}_{t})=\left\{0\right\}$.
\end{enumerate}
\end{Corollary}
\begin{proof} Notice that we have \eqref{gcont} so that we can repeat most of the arguments in the proof of 
Lemma \ref{lemp}. When
$\xi_t+\eta_t=0$, with $\xi_t,\eta_t\in L^0(G_t,{\cal F}_t)$ we have $\xi_t=\eta_t=0$.
\end{proof}
We also have
\begin{Corollary}
\eqref{A.1} implies (NA).
\end{Corollary}
\begin{proof}
The fact that \eqref{A.1} implies (NA) follows from the proof of Lemma \ref{timpl}. When there are no fixed costs then
$\mathbb{R}_{++}\subset G_t^0$ for $t\leq T$ and we can repeat the arguments of the proof of Lemma 3.4 of 
\cite{LT2}.
\end{proof}
\begin{Remark} The inverse inequality that is that (NA) implies \eqref{A.1} does not hold in the case of fixed 
plus concave
transaction costs as is shown in the Remark 3.5 of \cite{LT2}.
\end{Remark}

Comparing parts (b) of Lemmas \ref{lemp} and \ref{blemp} with Corollary \ref{glemp} we immediately obtain
\begin{Corollary}
If there are no fixed costs, then \eqref{tA.1} or \eqref{tbA.1} imply \eqref{A.1}.
\end{Corollary}
We have the following important result (see Lemma 3.2.8 of \cite{KS})
\begin{Proposition}\label{improp}
 Under \eqref{tA.1} the set ${\bar{R}}^0_T$ is closed in $L^0({\cal F}_T)$.
\end{Proposition}

For $(x,y)\in \mathbb{R}^2$ and $t\leq T$ define $Q^\epsilon_t(x,y)=(x,y\vee {-x \over
\underline{S}_t(\infty)-{\underline{S}_t(\infty)\over 2}\wedge \epsilon})$ for $x\leq 0$ and
$Q^\epsilon_t(x,y)=(x,y\vee {-x \over \underline{S}_t(\infty)+ \epsilon})$ for $x\geq 0$. Clearly 
$Q^\epsilon_t$ transforms $(x,y)\in
\mathbb{R}^2$ into $K_t$. For $(x,y)\in \bar{K}_t$ we have $\lim_{\epsilon \to 0}Q^\epsilon_t(x,y)=(x,y)$, 
$\mathbb{P}$ a.e..

\begin{Lemma}\label{closures}
Closure of ${\tilde{R}}^0_T$ in $L^0({\cal F}_T)$ coincides with the closure of ${\bar{R}}^0_T$, which under 
\eqref{tA.1} is equal to
${\bar{R}}^0_T$.
\end{Lemma}
\begin{proof} It suffices to show that for any element $z:=-\sum_{i=0}^T \xi_i\in {\bar{R}}^0_T$ such that 
$\xi_i\in
L^0({\bar{K}}_{i},{\cal F}_{i})$ there is a sequence $z_\epsilon\in {\tilde{R}}^0_T$ converging in probability 
to $z$. Let
$z_\epsilon=-\sum_{i=0}^T Q^\epsilon_i(\xi_i)$. Clearly $z_\epsilon\in {\tilde{R}}^0_T$ and $z_\epsilon \to z$, 
$\mathbb{P}$ a.e.,
which completes the proof.
\end{proof}
We shall need the following 
\begin{Lemma}
Concave hull of $L_t(x,y)$ defined as $\sup\left\{\sum \lambda_i L_t(x_i,y_i): \lambda_i\geq 0, \sum 
\lambda_i(x_i,y_i)=(x,y)\right\}$
coincides with $\bar{L}_t(x,y)$. Furthermore
\begin{equation}\label{impineqp}
\bar{L}_t(x,y)= \lim_{\lambda \to \infty} {L_t(\lambda (x,y))\over \lambda}
\end{equation}
and for
\begin{equation}\label{impineq}
\delta_t((x,y)):=\bar{L}_t(x,y)-L_t(x,y)=\bar{L}_t(0,y)-L_t((0,y))
\end{equation}
we have $0\leq \delta_t(x,y)$, and ${\delta_t(\lambda(x,y))\over \lambda}$  converges decreasingly to $0$ 
uniformly in $x$ and
uniformly in $y$ from compact subsets as $\lambda\to \infty$.
\end{Lemma}
\begin{proof}
The convergence \eqref{impineqp} follows from Proposition 2.6 (vi) of \cite{LT2}.
Notice furthermore that for $\lambda_1>\lambda_2$ by Corollary \ref{cor1.8} we have
$L_t(\lambda_1(x,y))\geq {\lambda_1 \over \lambda_2} L_t(\lambda_2(x,y))$ so that 
${\delta_t(\lambda_1(x,y))\over \lambda_1}\leq
{\delta_t(\lambda_2(x,y))\over \lambda_2}$. We have to show that the convergence of ${\delta_t(\lambda(x,y))\over 
\lambda}$ to $0$ is
uniform for bounded $y$.
Assume first that $M\geq y\geq 0$. For a given $\ve>0$ we want to show that  ${\delta_t(\lambda(x,y))\over 
\lambda}\leq \ve$ for a
sufficiently large $\lambda$. We have
\begin{equation}\label{fff}
\delta_t((x,y))=\mathbbm{1}_{y\geq 0} \left(y \underline{S}_t(\infty)-y(\underline{S}_t(y) + b_0)^+ \right) + 
\mathbbm{1}_{y<0}
\left(y(\overline{S}_t(\infty)-\overline{S}_t(-y)) + a_0 \right).
\end{equation}
Therefore
\begin{equation}
{\delta_t(\lambda(x,y)) \over \lambda}\leq y(\underline{S}_t(\infty)-(\underline{S}_t(\lambda y)) + {b_0\over 
\lambda}
\end{equation}
and  $ {b_0\over \lambda}\leq {\ve \over 2}$ and also $y(\underline{S}_t(\infty)-(\underline{S}_t(\lambda 
y))\leq {\ve \over 2}$,
since $y\leq M$, for a sufficiently large $\lambda$.
Consider now the case $0\geq y \geq -M$. Then
\begin{equation}
{\delta_t(\lambda(x,y)) \over \lambda}\leq y(\overline{S}_t(\infty)-\overline{S}_t(-\lambda y)) - {a_0\over 
\lambda} \leq -y (\overline{S}_t(-\lambda y)-\overline{S}_t(\infty))+ {-a_0\over \lambda}
\end{equation}
and similarly as above $-y (\overline{S}_t(-\lambda y)-\overline{S}_t(\infty))\leq {\ve \over 2}$ and
${-a_0\over \lambda}\leq {\ve \over 2}$ for a sufficiently large $\lambda$.
\end{proof}
\begin{Remark}
We can not expect that in general that
$\sup_y {\delta_t(\lambda(x,y))\over \lambda} $ converges to $0$ as $\lambda \to \infty$. In the case of our 
Example \ref{ex1} we have
$y(\underline{S}_t(\infty)-\underline{S}_t(\lambda y))={(\underline{S}_t(\infty)-\underline{S}_t(0))y \over 
(\lambda y + 1)^\alpha}\to
\infty$ when $y\to \infty$ and
$y(\overline{S}_t(\infty)-\overline{S}_t(-\lambda y))={-y(\overline{S}_t(0)-\overline{S}_t(\infty))\over 
(\lambda y + 1)^{\alpha}}\to
\infty$ when $y\to -\infty$.
\end{Remark}
\begin{Remark}\label{imfff}
 We can write explicitly the form of \eqref{fff}. Namely
for $y<0$ we have $\delta_t(x,y)=a_0 - y(\overline{S}_t(-y)-\overline{S}_t(\infty))$,
for $0\leq y\leq b_t^{-1}(0)$ we have $\delta_t(x,y)=y \underline{S}_t(\infty)$ and finally for $y\geq 
b_t^{-1}(0)$ we have
 $\delta_t(x,y)=-b_0+ y(\underline{S}_t(\infty)-\underline{S}_t(y))$.
\end{Remark}
\begin{Proposition}
Closures of $\widetilde{LV}^0_T$ and $\widebar{LV}_T^0$ in $L^0({\cal{F}}_T)$ coincide and are contained in the 
closure of $A_T^0$.
\end{Proposition}
\begin{proof}
 An element of $\widetilde{LV}^0_T$ is of the form $\bar{L}_T(\sum_{i=0}^T \xi_i)$ with $\xi_i\in 
 L^0(K_i,{\cal{F}}_i)$. Then by
 \eqref{impineq} we have
\begin{equation}
0\leq \bar{L}_T\left(\sum_{i=0}^T \xi_i\right)-{L_T(\lambda \sum_{i=0}^T \xi_i)\over \lambda}={\delta_T(\lambda 
\sum_{i=0}^T \xi_i)\over
\lambda}\to 0
\end{equation}
as $\lambda\to \infty$ in $L^0({\cal{F}}_T)$.  Therefore ${L_T(\lambda \sum_{i=0}^T \xi_i)\over \lambda}$ converges as $\lambda \to \infty$ 
to
$\bar{L}_T(\sum_{i=0}^T \xi_i)$ and  $\bar{L}_T(\sum_{i=0}^T \xi_i)$ is in the closure of $A_T^0$  in 
$L^0({\cal{F}}_T)$.
By Lemma \ref{closures} and continuity of $\bar{L}$ we obtain that the sets $\widetilde{LV}^0_T$ and 
$\widebar{LV}_T^0$ have the same
closure in $L^0({\cal{F}}_T)$, which completes the proof.
\end{proof}

\section{Strong arbitrage}

We say that we have {\it strong arbitrage} $(SA)$ on the market with solvent sets $(G_i)$ if there is 
$t\in\{0,1,2,...,T\}$ and set
$B_t \in \mathcal{F}_t$ such that $\mathbb{P}(B_t)>0$ and for portfolio process $V_n=\sum_{i=0}^n \xi_i$, with 
$\xi_i \in
L^0(-G_i,\mathcal{F}_i)$ we have $\xi_i=0$ for $i\leq t-1$, $\xi_i(\omega)=0$ for $i\geq t$ and $\omega \in 
B_t^c$ and $L_T(V_T)\geq
m_t>0$ on $B_t$, where $m_t \in L^0([0,\infty),\mathcal{F}_t)$.\newline
In the case when $G_i$ is replaced by $K_i$ or $\bar{K}_i$ we shall have {\it strong arbitrage} 
$(\widetilde{SA})$ or
$(\widebar{SA})$.\newline
We have {\it bounded arbitrage} $(BA)$ on the market with solvent sets $(G_i)$ when for portfolio process 
$V_n=\sum_{i=0}^n \xi_i$,
with $\xi_i \in L^0(-G_i,\mathcal{F}_i)$ there is $M>0$ such that for $\xi_i:=(\xi_i^1,\xi_i^2)$ we have
$|\xi_i|=|\xi_i^1|+|\xi_i^2|\leq M$ and $L_T(V_T)\geq 0$ with $\mathbb{P}\left\{L_T(V_T) > 0\right\}>0$. 
\newline
We say that we have {\it random bounded strong arbitrage} $(rbSA)$ on the market with solvent sets ($G_i$) if 
there is
$t\in\{0,1,2,...,T\}$ and set $B_t \in \mathcal{F}_t$ such that
$\mathbb{P}(B_t)>0$ and for portfolio process $V_n=\sum_{i=0}^n \xi_i$, with $\xi_i \in 
L^0(-G_i,\mathcal{F}_i)$ we have $\xi_i=0$ for
$i\leq t-1$, $\xi_i(\omega)=0$ for $i\geq t$ and $\omega \in B_t^c$ and $ |\xi_u(\omega) 
|=|\xi_u^1(\omega)|+|\xi_u^2(\omega)| \leq
\alpha_t(\omega)$ for $u\geq t$, $\omega \in B_t$ and $L_T(V_T)\geq m_t>0$ on $B_t$, where $m_t, \alpha_t \in
L^0((0,\infty),\mathcal{F}_t)$.\newline
In the case when there is $M>0$ such that $\alpha_t(\omega)\leq M$ we have {\it bounded strong arbitrage} 
$(BSA)$.
Similarly as above, when $G_i$ are replaced by $K_i$, $\bar{K}_i$ we shall have ($\widetilde{rbSA}$) and 
($\widebar{rbSA}$) or
($\widetilde{BSA}$) and ($\widebar{BSA}$) respectively.\newline
{\it Absence of arbitrage} will be denoted by $(NSA)$, $(\widetilde{NSA})$, $(\widebar{NSA})$ or $(NrbSA)$, 
$(\widetilde{NrbSA})$,
$(\widebar{NrbSA})$. It is rather clear that
\begin{eqnarray*}
&&(SA)\implies(A), \quad (\widetilde{SA})\implies (\widetilde{A}) \quad\text{and}\quad (\widebar{SA})\implies 
(\widebar{A}) \\
&&(rbSA)\implies(A), \quad (\widetilde{rbSA})\implies (\widetilde{A}) \quad \text{and} \quad 
(\widebar{rbSA})\implies (\widebar{A})
\end{eqnarray*}
and conversely
\begin{eqnarray*}
&&(NA)\implies(NSA), \quad (\widetilde{NA})\implies (\widetilde{NSA}), \quad (\widebar{NA})\implies 
(\widebar{NSA}) \\
&&(NA)\implies(NrbSA), \quad (\widetilde{NA})\implies (\widetilde{NrbSA}) \text{  and  } 
(\widebar{NA})\implies (\widebar{NrbSA})
\end{eqnarray*}
and similar implications hold in the case of bounded portfolio strategies.
Consider now an example being a continuation of Example \ref{ex1}
\begin{Example}
Assume we have the following bid and ask curves as in \eqref{exbidp} and \eqref{exaskp} 
\begin{equation}
\underline{S}_t(m)=\underline{S}_t(0) \underline{d}_t(m),
\end{equation} 
where $\underline{d}_t(m)={(1+p_t)(m+1)^\alpha - p_t \over (M+1)^\alpha}$ with deterministic $p_t\geq 0$,
\begin{equation}
\overline{S}_t(l)=\overline{S}_t(0)\overline{d}_t(l),
\end{equation}
where $\overline{d}_t(l)={(1-q_t)(l+1)^\alpha + q_t \over (l+1)^\alpha}$ with deterministic $q_t\in [0,1)$, assuming additionally that $(1-q_t)\overline{S}_t(0)\geq (1+p_t) \underline{S}_t(0)$. Let now 
$\underline{S}_t(0)=(1+e_t)\underline{S}_{t-1}(0)$, $\overline{S}_t(0)=(1+f_t)\overline{S}_{t-1}(0)$ for $t=1,2,\ldots$, with ${\cal F}_t$ measurable random variable $e_t,f_t$ taking values in $(-1,\infty)$.  
Assume furthermore that there is a set $B_1\in {\cal F}_1$, $\mathbb{P}(B_1)>0$ and an ${\cal F}_1$ measurable random variable $\bar{e}_1$ such that $e_2\geq \bar{e}_1$ on $B_1$ and for some   ${\cal F}_1$ nonnegative random variable  $l$
\begin{equation}
(1+\bar{e}_1)\underline{S}_1(0)\underline{d}_2(l)-\overline{S}_1(0)\overline{d}_1(l):=m_1>0 
\end{equation}
on the set $B_1$. 
The strategy to buy $l$ assets at time $1$ on the event $B_1$ and sell them at time $2$ gives a strong arbitrage. 
If $l\overline{S}_1(l)$, $l\underline{S}_2(l)$ and $l$ are bounded we have bounded arbitrage, and when they are bounded by an ${\cal F}_1$ measurable random variable we have random bounded strong arbitrage.
\end{Example}  
Consider the following two conditions
\begin{enumerate}
\item[$(gL_0)$]  $\lim_{\lambda \to \infty} 
    \sup_{y}\left\|\frac{\delta_t(\lambda(0,y))}{\lambda}\right\|_{\infty}=0$, \label{gL}
\item[$(L_0)$] $\lim_{\lambda \to \infty} \sup_{|y|\leq 
    M}\left\|\frac{\delta_t(\lambda(0,y))}{\lambda}\right\|_{\infty}=0$ for any
    $M \in (0,\infty)$, \label{L}
\end{enumerate}
where $\|\cdot \|_{\infty}$ stands for $L^{\infty}$ norm.

\begin{Remark} Condition $(gL_0)$ is satisfied when $(G_i)$ correspond to fixed plus proportional transaction 
costs i.e.
$a_t(l)=a_0+l\overline{S}_t$, $b_t(m)=b_0+m\underline{S}_t$. Then clearly $\delta_t(0,y)\leq 
\max\left\{a_0,-b_0 \right\}$.
\end{Remark}

\begin{Lemma}\label{lem15}
When $\|\overline{S}_t(0)\|_{\infty}<\infty$,  $\|\underline{S}_t(\infty)-\underline{S}_t(\lambda 
y)\|_{\infty}\to 0$ for $y > 0$ and
$\|\overline{S}_t(\infty)-\overline{S}_t(-\lambda y)\|_{\infty}\to 0$ for $y < 0$ as $\lambda \to \infty$ then 
condition $(L_0)$ is
satisfied.
\end{Lemma}
\begin{proof}
There is $\lambda_n\to \infty$ such that
\begin{equation}
\lim_{\lambda \to \infty} \sup_{|y|\leq M}\left\|\frac{\delta_t(\lambda(0,y))}{\lambda}\right\|_{\infty}=
\lim_{n \to \infty} \sup_{|y|\leq M}\left\|\frac{\delta_t(\lambda_n(0,y))}{\lambda_n}\right\|_{\infty}.
\end{equation}
For each $\epsilon>0$ and  $\lambda_n$ there is $y_n$ such that
\begin{equation}\label{api}
\sup_{|y|\leq M}\left\|\frac{\delta_t(\lambda_n(0,y))}{\lambda_n}\right\|_{\infty}-
\left\|\frac{\delta_t(\lambda_n(0,y_n))}{\lambda_n}\right\|_{\infty}\leq \epsilon.
\end{equation}
We have two cases: either there is a subsequence $n_k$ such that $y_{n_k}\to 0$ as $k\to \infty$ or
there is $c>0$ such that $|y_n|\geq c$ for sufficiently large $n$. In the first case since 
$\max\left\{\underline{S}_t(m),\overline{S}_t(l)\right\}\leq \overline{S}_t(0)$ for $l,m\geq 0$ we have
$\|y_{n_k}(\underline{S}_t(\infty)-\underline{S}_t(\lambda_{n_k} y_{n_k}))\|_{\infty}\to 0
$
and
$\|y_{n_k}(\overline{S}_t(\infty)-\overline{S}_t(-\lambda_{n_k} y_{n_k}))\|_{\infty}\to 0
$
as $k\to \infty$, while in the second case
$\|y_{n}(\underline{S}_t(\infty)-\underline{S}_t(\lambda_{n} y_{n}))\|_{\infty}\to 0$
and
$\|y_{n}(\overline{S}_t(\infty)-\overline{S}_t(-\lambda_{n} y_{n}))\|_{\infty}\to 0$
as $n\to \infty$. Consequently using Remark \ref{imfff} and \eqref{api} we obtain
\begin{equation}
\lim_{\lambda \to \infty} \sup_{|y|\leq M}\left\|\frac{\delta_t(\lambda(0,y))}{\lambda}\right\|_{\infty}\leq 
\epsilon,
\end{equation}
which taking into account that $\epsilon$ could be chosen arbitrarily small completes the proof.
\end{proof}
\begin{Remark}
In the case of Example \ref{ex1} under $\|\overline{S}_t(0)\|_{\infty}<\infty$ we have that
$\|\underline{S}_t(\infty)-\underline{S}_t(\lambda y)\|_{\infty}\to 0$ for $y > 0$ and
$\|\overline{S}_t(\infty)-\overline{S}_t(-\lambda y)\|_{\infty}\to 0$ for $y < 0$ as $\lambda \to \infty$ and 
consequently by Lemma
\ref{lem15} condition $(L^0)$ is satisfied.
\end{Remark}

\begin{Theorem}\label{thm:3.2}
Under $(gL_0)$ condition ($NSA$) is equivalent to ($\widebar{NSA}$), while when $\|\overline{S}_T(0)\|_\infty 
<\infty$ under $(L_0)$
condition ($NrbSA$) is equivalent to ($\widebar{NrbSA}$).
\end{Theorem}
\begin{proof}
Since $(\widebar{NSA})\implies (NSA)$ to have inverse implication it is sufficient to show that 
$(\widebar{SA})\implies (SA)$.
Suppose there is $t \in \{0,...,T\}$, set $B_t \in \mathcal{F}_t$, $\mathbb{P}(B_t)>0$ and portfolio 
$\tilde{V}_u=\sum_{i=0}^u
\tilde{\xi}_i$ such that $\tilde{\xi}_i=0$ on $B_t^c$ for $i=0,1,2,...,T$, while $\tilde{\xi}_i=0$ on $B_t$ for 
$i=1,2,...,t-1$ and
$\tilde{\xi}_i \in L^0(-\bar{K}_i,\mathcal{F}_i)$ for $i=t,t+1,...,T$. Moreover 
$\bar{L}_T(\tilde{V}_T)>m_t>0$ on $B_t$
with $m_t \in L^0((0,\infty),\mathcal{F}_t)$.
We may assume that $m_t(\omega)\geq \epsilon >0$ on $B_t$ (replacing if necessary $m_t$ by 
$\mathbbm{1}_{\{m_t\geq \epsilon\}}m_t$ and
$B_t$ by $B_t \cap \{m_t \geq \epsilon \}$). Another assumption is that $\tilde{V}_T$ is in the form
$\bar{L}_T(\tilde{V}_T)\cdot e_1$, where $e=(0,1)$. This requires suitable explanation:\newline
Notice first that $\bar{L}_T(\bar{L}_T(\tilde{V}_T)\cdot e_1)=\bar{L}_T(\tilde{V}_T)\geq 
0$ and therefore
$\bar{L}_T(\tilde{V}_T)\cdot e_1 \in L^0(\bar{K}_T,\mathcal{F}_T)$. Let
$\zeta_T=\tilde{V}_T-\bar{L}_T(\tilde{V}_T)\cdot e_1$. Since $\bar{L}_T(\zeta_T)=0$ we have 
that $\zeta_T \in
L^0(\bar{K}_T,\mathcal{F}_T)$. Therefore
\begin{equation}\label{eeeq}
L^0(\bar{K}_T,\mathcal{F}_T) \ni \bar{L}_T(\tilde{V}_T)\cdot
e_1=\tilde{V}_T-\zeta_T=\sum_{i=0}^{T-1}\tilde{\xi}_i+(\tilde{\xi}_T-\zeta_T)
\end{equation}
$\tilde{\xi}_i \in L^0(-\bar{K}_i, \mathcal{F}_i)$ for $i\leq T-1$, $(\tilde{\xi}_T-\zeta_T)\in 
L^0(-\bar{K}_T,
\mathcal{F}_T)$ and when we have $(\widebar{SA})$ for $\tilde{V}_T$ then also for 
$\bar{L}_T(\tilde{V}_T)\cdot
e_1$.\newline
We assume therefore that $\tilde{V}_T =\bar{L}_T(\tilde{V}_T)\cdot e_1$, where $\tilde{V}_T=\sum_{i=0}^T 
\tilde{\xi}_i$ with
$\tilde{\xi}_i\in L^0(-\bar{K}_i, \mathcal{F}_i)$     and $m_t\geq\epsilon$ on $B_t$. For $k>1$ let
\begin{equation}
\xi_r^k=k\tilde{\xi}_r+L_r(-k\tilde{\xi}_r)e_1 \quad \text{for } t\leq r\leq T
\end{equation}
and define $V_u^k=\sum_{r=1}^u\xi_r^k \mathbbm{1}_{B_t}$ for $u\geq t$.\newline
We have that
\begin{equation}
L_r(-\xi_r^k)=L_r\left(-k\tilde{\xi}_r-L_r(-k\tilde{\xi}_r)e_1\right)=L_r(-k\tilde{\xi}_r)-L_r(-k\tilde{\xi}_r)=0.
\end{equation}
So that $-\xi_r^k \in L^0(G_r,\mathcal{F}_r)$ i.e. $\xi_r^k \in L^0(-G_r,\mathcal{F}_r)$ and $V^k$ is a $G$ 
portfolio.\newline
Now
\begin{equation}\label{rp1}
L_r(-k\tilde{\xi}_r)=\bar{L}_r(-k\tilde{\xi}_r)-\delta_r(-k\tilde{\xi}_r)\geq -\delta_r(-k\tilde{\xi}_r)
\end{equation}
since $\bar{L}_r(-k\tilde{\xi}_r)\geq 0$ and we have on $B_t$
\begin{equation}
L_T(V_T^k)=L_T\left(k\cdot
\tilde{V}_T+\sum_{r=t}^TL_r(-k\tilde{\xi}_r)e_1\right)=L_T(k\tilde{V}_T)+\sum_{r=t}^TL_r(-k\tilde{\xi}_r).
\end{equation}
Since $\tilde{V}_T=\bar{L}_T(\tilde{V}_T)e_1$ by Corollary \ref{cor1.8}  and \eqref{rp1}
\begin{eqnarray}
L_T(V_T^k)&\geq& k
L_T(\tilde{V}_T)-\sum_{r=t}^T\delta_r(-k\tilde{\xi}_r)=k\left(L_T(\tilde{V}_T)-\sum_{r=t}^T\frac{\delta_r(-k\tilde{\xi}_r)}{k}
\right)\geq \nonumber \\
&&k\left(\epsilon - \sum_{r=t}^T\sup_y \frac{\delta_r(-k(0,y))}{k} \right)\geq 
k\left(\epsilon - \sum_{r=t}^T \sup_y \left\| \frac{\delta_r(-k(0,y))}{k}\right\|_{\infty} \right)>0
\end{eqnarray}
for a sufficiently large $k$, $\mathbb{P}$ a.e., so that we have $(SA)$. \newline
Assume now $L_0$ and $(\widebar{rbSA})$. Then $|\tilde{\xi}_u(\omega)|\leq \alpha_{t}$ for $T-1\geq u\geq t$ 
and by \eqref{eeeq} we
have that $\tilde{\xi}_T$ is now equal to $\tilde{\xi}_T-\zeta_T$ and
$|\tilde{\xi}_T-\zeta_T|\leq |\sum_{i=t}^{T-1}\tilde{\xi}_i|+|\bar{L}_T(\tilde{V}_T)|\leq T\alpha_t (1+
\|\overline{S}_T(0)\|_\infty)$. Therefore with an abuse of  notation we can write that there is $\mathcal{F}_t$ 
measurable $\alpha_t$
such that for all $0\leq i \leq T$ we have $|\tilde{\xi}_i|\leq \alpha_t$ and
\begin{equation}
\sum_{r=t}^T\frac{\delta_r(-k\tilde{\xi}_r)}{k}\leq \sum_{r=t}^T\sup_{|y| \leq 
\alpha_t}\frac{\delta_r(-k(0,y))}{k}\leq
\sum_{r=t}^T\sup_{|y| \leq \alpha_t}\frac{\|\delta_r(-k(0,y))\|_{\infty}}{k}.
\end{equation}
We choose random $k$ such that $\forall_{r\geq t}\ \sup_{y \leq
\alpha_t}\frac{\|\delta_r(-k(0,y))\|_{\infty}}{k}\leq\frac{\epsilon}{2(T-t+1)}$. Clearly $k$ is $\mathcal{F}_t$ 
measurable. Then for
portfolio $V_T^k$ with random $k$ we have
\begin{equation}\label{impineqq}
L_t(V_T^k)\geq k\left(\epsilon - \sum_{r=t}^T\frac{\delta_r(-k\tilde{\xi}_r)}{k}  \right)\geq k\left(\epsilon - 
\frac{\epsilon}{2}
\right)=k(\omega)\frac{\epsilon}{2}.
\end{equation}
Consequently $(\widebar{rbSA})$ implies $(rbSA)$, and since $(\widebar{NrbSA})$ implies $(NrbSA)$ we have 
that $(\widebar{NrbSA})$
is equivalent to $(NrbSA)$.

\end{proof}

\begin{Remark}
Theorem \ref{thm:3.2} corrects Theorem 6.4 of \cite{LT2}. In the case of the assumption $(L_0)$ we had to 
introduce the notion of
bounded strong arbitrage since otherwise we couldn't get \eqref{impineqq}. The same problem appears in the 
proof of Theorem 6.6 of
\cite{LT2}, which can not be adapted in the case of $(L_0)$.
\end{Remark}
An analysis of the proof of Theorem \ref{thm:3.2} shows that

\begin{Corollary}
Assuming that $\|\overline{S}_t(0)\|_\infty <\infty$ and $(L_0)$ holds we have that $(\widebar{BSA})$ is 
equivalent to  $(BSA)$.
\end{Corollary}
\begin{proof}
Notice that clearly $(BSA)$ implies $(\widebar{BSA})$ and we therefore have to show that $(\widebar{BSA})$ 
implies $(BSA)$. The
proof then follows ideas of the second part of the proof of Theorem \ref{thm:3.2}.
\end{proof}

\noindent
\textbf{Assumption $(L_1)$}: there is an ${\cal F}_t$ adapted positive random variable $c_t(y)$, $y\in 
\mathbb{R}$ such that
$L_t(0,y)\leq\bar{L}_t(0,y)-c_t(y)$ for $y\neq 0$ and $t \in \{0,1,...,T\}$.

\begin{Remark} Assumption $(L_1)$ is not restrictive. It is satisfied for fixed plus proportional transaction 
costs, fixed plus
concave transaction costs or for strictly concave transaction costs. It is not satisfied only when there is $t$ 
such that  the
boundary of $G_t$ contains either halfline $y={-x\over \overline{S}_t(\infty)}$ with $x\geq 0$ or $y={-x \over
\underline{S}_t(\infty)}$ for $x\leq 0$.
\end{Remark}

\begin{Theorem}\label{thm:3.3} Under ($gL_0$) and ($L_1$) we have that $(A) \equiv (SA)\equiv (\widebar{SA})$.
\end{Theorem}
\begin{proof}
We want to show that $(A) \implies (\widebar{SA})$. Assume $(A)$. Then there is $V_t=\sum_{u=0}^{t}\xi_u$, 
$\xi_u \in L^0
(-G_u,\mathcal{F}_u)$ such that $L_T(V_T)\geq0$ and $\mathbb{P}\{L_T(V_T)>0\}>0$. Let $\hat{\xi}_u=\xi_u 
\mathbbm{1}_{\xi_u \notin
-\mathbb{R}_+e_1\backslash\{0\}}$. For $\tilde{V}_t=\sum_{u=0}^t \hat{\xi}_u$ we have
\begin{eqnarray}
&&
L_t(V_t)=L_t\left( \tilde{V}_t+\sum_{u=0}^t \xi_u\mathbbm{1}_{\xi_u \in -\mathbb{R}_+e_1}\right)=L_t\left(
\tilde{V}_t+\sum_{u=0}^tL_u(\xi_u)e_1\mathbbm{1}_{\xi_u \in -\mathbb{R}_+e_1}\right)= \nonumber \\
&&L_t(\tilde{V}_t)+\sum_{u=0}^tL_u(\xi_u)\mathbbm{1}_{\xi_u \in -\mathbb{R}_+e_1}\leq L_t(\tilde{V}_t)
\end{eqnarray}
since for $\xi_u=(\xi_u^1,\xi_u^2)$ we have $\xi_u\mathbbm{1}_{\xi_u \in 
-\mathbb{R}_+e_1}=\xi_u^1\mathbbm{1}_{\xi_u \in
-\mathbb{R}_+e_1}e_1$ and  $\xi_u^1\mathbbm{1}_{\xi_u \in -\mathbb{R}_+e_1} =L_u(\xi_u)\mathbbm{1}_{\xi_u \in 
-\mathbb{R}_+e_1}\leq0$.
Consequently also $L_T(\tilde{V}_T)\geq0$ and $\mathbb{P}\{L_T(\tilde{V}_T)>0\}>0$. In what follows we 
shall assume that
$V_t=\sum_{u=0}^t\xi_u$ and $\xi_u \notin -\mathbb{R}_+e_1\backslash\{0\}$. Let 
$t^{\star}=\min\{t\geq0:\mathbb{P}\{\xi_t
\neq 0\}>0\}$ ($0$ means $(0,0)$ in our case).  Set $B_{t^{\star}}=\{\xi_{t^{\star}} \neq 0\}$. We 
have that
$\mathbb{P}(B_{t^{\star}})>0$. We may also assume that $V_T=L_T(V_T)\cdot e_1$, since as in the proof of 
Theorem \ref{thm:3.2} for
$\zeta_T=V_T-L_T(V_T)\cdot e_1$ we have $L_T(\zeta_T)=0$, so that  $\zeta_T \in L^0(G_T,\mathcal{F}_T)$
\begin{equation*}
L_T(V_T)\cdot e_1=V_T-\zeta_T=\sum_{u=0}^{T-1}\xi_u +\xi_T-\zeta_T
\end{equation*}
and $L_T(L_T(V_T)\cdot e_1)=L_T(V_T)\geq0$, $\mathbb{P}\{L_T(L_T(V_T)\cdot 
e_1)>0\}=\mathbb{P}\{L_T(V_T)>0\}>0$.\newline
On $B_{t^{\star}}$ we have (using $(L_1)$)
\begin{equation*}
\bar{L}_{t^{\star}}(-\xi_{t^{\star}})\geq 
L_{t^{\star}}(-\xi_{t^{\star}})+c_{t^{\star}}(\xi_{t^{\star}}^2)\geq
c_{t^{\star}}(\xi_{t^{\star}}^2) \quad (\text{since }L_{t^{\star}}(-\xi_{t^{\star}})\geq 0).
\end{equation*}
Let 
$\tilde{V}_T=\sum_{t=t^{\star}}^{T}\left(\xi_t+\bar{L}_t(-\xi_t)e_1\right)\mathbbm{1}_{B_{t^{\star}}}$ 
and
$\tilde{\xi}_t=\left(\xi_t+\bar{L}_t(-\xi_t)e_1\right)\mathbbm{1}_{B_{t^{\star}}}$. We have
\begin{equation*}
\bar{L}_t(-\tilde{\xi}_t)=\bar{L}_t\left(-\xi_t-\bar{L}_t(-\xi_t)e_1\right)=0 \quad \text{on 
}B_{t^{\star}}
\end{equation*}
so that $-\tilde{\xi}_t \in \bar{K}_t$ and $\tilde{\xi}_t \in -\bar{K}_t$ on $B_{t^{\star}}$. 
Therefore on $B_{t^{\star}}$
we have, taking into account that $V_T=L_T(V_T)\cdot e_1$
\begin{equation*}
\bar{L}_T(\tilde{V}_T)=\bar{L}_T\left(V_T+\sum_{t=t^{\star}}^T\bar{L}_t(-\xi_t)e_1
\right)=L_T(V_T)+\sum_{t=t^{\star}}^T\bar{L}_t(-\xi_t)\geq \sum_{t=t^{\star}}^T\bar{L}_t(-\xi_t) \geq
\bar{L}_{t^{\star}}(-\xi_{t^{\star}})\geq c_{t^{\star}}(\xi_{t^{\star}}^2)
\end{equation*}
(since $-\xi_t \in L^0(-G_t,\mathcal{F}_t)\subset L^0(-\bar{K}_t, \mathcal{F}_t)$ we have that 
$\bar{L}_t(-\xi_t)\geq 0$).
Consequently we have $(\widebar{SA})$ for portfolio $\tilde{V}_T$.\newline
Since by Theorem \ref{thm:3.2} $(\widebar{SA}) \equiv (SA)$ we therefore have $(A)\implies (SA)$. The 
implication $(SA) \implies (A)$
is obvious and we finally obtain that $(A)\equiv (SA)$.
\end{proof}
Following the proof of Theorem \ref{thm:3.3} we see that

\begin{Corollary}
Under $(L_1)$ we have $(A) \Longrightarrow (\widebar{SA})$.
\end{Corollary}
We can also adapt the proof of Theorem \ref{thm:3.3} to obtain the following result

\begin{Corollary}
Assuming that $\|\overline{S}_T(0)\|_\infty <\infty$ under $(L_1)$ we have that $(BA)$ implies 
$\widebar{BSA}$. Assuming additionally
$(L_0)$ we obtain that $(BA)$ is equivalent to $(BSA)$ and also to
$\widebar{BSA}$.
\end{Corollary}

\section{Asymptotic arbitrage}
Assume we are given two markets $M_1$ and $M_2$ with bid and ask curves $(\underline{S}_t^1, \overline{S}_t^1)$ 
and 
$(\underline{S}_t^2, \overline{S}_t^2)$, adapted to filtration ${\cal F}_t$ respectively. We also assume that 
these bid and ask curves 
satisfy all conditions (a1)-(a5)  formulated in the Introduction. By analogy to \eqref{liquid} and \eqref{solv} 
we define liquidation 
functions $L_t^1$ and $L_t^2$ and solvent sets $G_t^1$ and $G_t^2$ respectively for $t=0,1,\ldots, T$.

We say that {\it arbitrage on $M_1$ market implies weak asymptotic arbitrage on $M_2$ market} if there is a sequence 
$\zeta_T^n=\sum_{t=0}^T \xi_t^n$ with $\xi_t^n\in -G_t^1$ of 
portfolios taking values in $R_T^0(G^1):=\sum_{t=0}^T L^0(-G_t^1,{\cal{F}}_t)$ for which for each $n\in 
\mathbb{N}$  we have an 
arbitrage i.e. $\mathbb{P}\left\{L_T^1(\zeta_T^n)\geq 0\right\}=1$ and $\mathbb{P}\left\{L_T^1(\zeta_T^n) > 
0\right\}>0$ and 
also
 $\lim_{n\to \infty}\mathbb{P}\left\{\forall_t  \ \xi_t^n\in -G_t^2\right\}=1$
$\liminf_{n\to \infty}\mathbb{P}\left\{L_T^2(\zeta_T^n)\geq 0\right\}=1$ and
$\liminf_{n\to \infty}\mathbb{P}\left\{L_T^2(\zeta_T^n) > 0\right\}>0$.
We also use notation $R_T^0(G^2):=\sum_{t=0}^T L^0(-G_t^2,{\cal{F}}_t)$.

We say that {\it asymptotic arbitrage on $M_1$ market implies weak asymptotic arbitrage on $M_2$ market} if there is a sequence $\zeta_T^n=\sum_{t=0}^T \xi_t^n$ with $\xi_t^n\in -G_t^1$ of portfolios taking values in $R_T^0(G^1):=\sum_{t=0}^T L^0(-G_t^1,{\cal{F}}_t)$ such that $\liminf_{n\to \infty}\mathbb{P}\left\{L_T^1(\zeta_T^n)\geq 0\right\}=1$ and $\liminf_{n\to \infty}\mathbb{P}\left\{L_T^1(\zeta_T^n) > 
0\right\}>0$ and also $\lim_{n\to \infty} \mathbb{P}\left\{\forall_t \ \xi_t^n\in -G_t^2\right\}=1$,  
$\liminf_{n\to \infty}\mathbb{P}\left\{L_T^2(\zeta_T^n)\geq 0\right\}=1$ and
$\liminf_{n\to \infty}\mathbb{P}\left\{L_T^2(\zeta_T^n) > 0\right\}>0$.

We say that {\it asymptotic arbitrage on $M_1$ market implies asymptotic arbitrage on $M_2$ market} if there is a sequence $\zeta_T^n=\sum_{t=0}^T \xi_t^n$ with $\xi_t^n\in -G_t^1$ of portfolios taking values in $R_T^0(G^1):=\sum_{t=0}^T L^0(-G_t^1,{\cal{F}}_t)$ 
such that 

$\liminf_{n\to \infty}\mathbb{P}\left\{L_T^1(\zeta_T^n)\geq 0\right\}=1$ and $\liminf_{n\to \infty}\mathbb{P}\left\{L_T^1(\zeta_T^n) > 
0\right\}>0$ and 

 $\mathbb{P}\left\{\forall_t \ \xi_t^n \in -G_t^2)\right\}=1$, 
$\liminf_{n\to \infty}\mathbb{P}\left\{L_T^2(\zeta_T^n)\geq 0\right\}=1$ and
$\liminf_{n\to \infty}\mathbb{P}\left\{L_T^2(\zeta_T^n) > 0\right\}>0$.

Let $M^\alpha$ be a market with proportional transaction costs with bid prices 
$(\underline{S}_t=(1-\alpha)\underline{S}_t(\infty))$ 
and ask prices $(\overline{S}_t=(1+\alpha)\overline{S}_t(\infty))$, where $0<\alpha \leq \min\left\{ 
{\overline{S}_t(0) \over \overline{S}_t(\infty)}-1, 1-{\underline{S}_t(0) \over 
\underline{S}_t(\infty)}\right\}$, $\mathbb{P}$ a.s. (note that $\overline{S}_t\leq \overline{S}_t(0)$ and 
$\underline{S}_t\geq \underline{S}_t(0)$). For such market define
\begin{equation}
{L}_t^\alpha(x,y)=\mathbbm{1}_{y\geq 0} \left(x+y (1-\alpha)\underline{S}_t(\infty)\right) + 
\mathbbm{1}_{y<0} \left(x+y
(1+\alpha)\overline{S}_t(\infty)\right)=x+{L}_t^\alpha(0,y). 
\end{equation}
and the solvent set $\bar{K}_t^\alpha:=\left\{(x,y): L_t^\alpha (x,y)\geq 0\right\}$.  
We have 

\begin{Lemma}\label{Lem4} Assume that 
\begin{enumerate}
\item[(B1)] $\|\overline{S}_t(\infty)\|_\infty \leq S<\infty $ for $t=0,1,\ldots,T$ 
\item[(B2)] $\max\left\{\|\underline{S}_t(z)-\underline{S}_t(\infty)\|_{\infty},
\|\overline{S}_t(z)-\overline{S}_t(\infty)\|_{\infty}\right\}\leq\beta(z)$
where $\beta(z)$ decreases to $0$ as $z\to \infty$, and for a given $\ve>0$ such that $\alpha 
S>\varepsilon>0$ there is $M(\varepsilon)$ such that for $z\geq 
M(\varepsilon)$ we have $\beta(z)\leq \varepsilon$. 
\end{enumerate}
 Then line $y={-x \over (1+\alpha) \overline{S}_t(\infty)}$ is above the curve 
$y=-a_t^{-1}(x)$  for $x\geq 0$ and  $y\leq 0$ when 
$y\leq
\min\left\{-M(\varepsilon), {a_0 \over \varepsilon-\alpha S}\right\}$.  Furthermore line $y={-x \over 
(1-\alpha)\underline{S}_t(\infty)}$ is above the curve $y=b_t^{-1}(-x)$ for $x\leq 0$ and $y\geq 0$ when
$y\geq \max \left\{M(\varepsilon), {-b_0 \over \alpha S}\right\}$. 
\end{Lemma}
\begin{proof}
 Let $y\geq -a_t^{-1}(x)$. Clearly $x \geq a_t(-y)$ and
we have $x\geq a_0-y\overline{S}_t(-y)$ and $x=-y(1+\alpha)\overline{S}_t(\infty)$. Consequently 
$a_0-y\overline{S}_t(-y)\leq 
-y(1+\alpha)\overline{S}_t(\infty)$ and for $-y\geq M(\varepsilon)$ 
\begin{equation}a_0 \leq y(\overline{S}_t(-y)-(1+\alpha)\overline{S}_t(\infty))\leq -y(-\varepsilon+\alpha 
\overline{S}_t(\infty))\leq -y(-\varepsilon+\alpha S)
\end{equation}
and $-y\geq {a_0\over -\varepsilon+\alpha S}$.
Assume now that $y\geq b_t^{-1}(-x)$. Then $-x\leq b_0+y\underline{S}_t(y)$ and since $-x=y(1-\alpha) 
\underline{S}_t(\infty)$  we 
have for $y\geq M$
\begin{equation}
b_0\geq y((1-\alpha)\underline{S}_t(\infty)- \underline{S}_t(y))\geq y (-\alpha \underline{S}_t(\infty))\geq 
-y \alpha S.
\end{equation}
Therefore $y\geq {-b_0 \over \alpha S}$.
\end{proof}
We have
\begin{Proposition}\label{prop3}
Under (B1) and (B2) 
arbitrage on $M^\alpha$ market implies weak asymptotic arbitrage on the $(G_t)$ market (with bid and ask curves 
$(\underline{S}_t, \overline{S}_t)$).
\end{Proposition}
\begin{proof}
Arbitrage on $M^\alpha$ market means that there is  $\zeta_T \in R_T^0(K^\alpha):=\sum_{t=0}^T 
L^0(-\bar{K}_t^\alpha,{\cal{F}}_t)$
such that  $\mathbb{P}\left\{L_T^\alpha(\zeta_T)\geq 0\right\}=1$ and $\mathbb{P}\left\{L_T^\alpha(\zeta_T) > 
0\right\}>0$ and 
$\zeta_T=\sum_{i=0}^T \xi_i$, where $\xi_i=(\xi_i^1,\xi_i^2)$ is ${\cal F}_i$ measurable and takes values in 
$-\bar{K}_i^\alpha$. Consider  
 $\varepsilon>0$ such that $\min\left\{{1\over T},\mathbb{P}\left\{L_T^\alpha(\zeta_T) > 
0\right\}\right\}>2\varepsilon>0$. For $t=0,1,\ldots,T$ there is a deterministic $d_t>0$ such that 
$\mathbb{P}\left\{|\xi_t^2|\geq d_t \ or \ \xi_t^2=0\right\}\geq 1-\varepsilon$. Then by Lemma \ref{Lem4} for each  $t=0,1,\ldots,T$ there is a deterministic $n_t$ such that 
$n_t \xi_t\in -G_t$ on the set $\Omega_t=\left\{|\xi_t^2|\geq d_t \ or \ \xi_t^2=0\right\}$, and for 
$n=n_0+n_1+\ldots + n_T$ we also 
have $n \xi_t \in -G_t$ for $t=0,1,\ldots,T$ on the set $\cap_{t=0}^T\Omega_t$. Clearly 
$\mathbb{P}\left\{\cap_{t=0}^T\Omega_t\right\}\geq 1-T\varepsilon$ 
and for $n\zeta_T$ we have an arbitrage on the market $M^\alpha$. Moreover  $\mathbb{P}\left\{\forall_t \ n\xi_t\in -G_t\right\}\geq 1-T\varepsilon$ and it is satisfied also for any $\bar{n}\geq n$. Since 
$\mathbb{P}\left\{L_T^\alpha(n\zeta_T)\geq 0\right\}=1$ and 
$\mathbb{P}\left\{L_T^\alpha(n\zeta_T) > 0\right\}>0$ there is a deterministic $d>0$ such that 
$\mathbb{P}\left\{\Omega(d)\right\}\geq 1-\varepsilon$ 
with $\Omega(d)=\Omega^1(d)\cup \Omega^2(d) \cup \Omega^3(d)$, where $\Omega^1(d)=\left\{|n\zeta_T^2|\geq d \right\}$, $\Omega^2(d)=\left\{\zeta_T^2=0,\zeta_T^1>0\right\}$ and $\Omega^3(d)=\left\{\zeta_T=0\right\}$.  
By Lemma \ref{Lem4} there is a 
deterministic $n(d)$ such that $n(d)n\zeta_T\in G_T^0\cup\left\{0\right\}$ on $\Omega^1(d)$ and consequently $L_T(n(d)n\zeta_T)>0$ on $\Omega^1(d)$. Moreover $L_T(n(d)n\zeta_T)>0$ on $\Omega^2(d)$ and $L_T(n(d)n\zeta_T)=0$ on 
 $\Omega^3(d)$. Summarizing we have $\mathbb{P}\left\{\forall_t \ n(d)n\xi_t \in -G_t\right\}\geq 1-T\varepsilon$, $\mathbb{P}\left\{L_T(n(d)n\zeta_T)\geq 0\right\}=\mathbb{P}\left\{\Omega(d)\right\}\geq 1-\varepsilon$ and 
\begin{eqnarray}
&&\mathbb{P}\left\{L_T(n(d)n\zeta_T) > 0\right\}=
\mathbb{P}\left\{\Omega^1(d)\cup \Omega^2(d)\right\}\geq 1-\varepsilon - \mathbb{P}\left\{\Omega^3(d)\right\}\geq  \nonumber \\
&& 1-\varepsilon - \mathbb{P}\left\{L_T^\alpha(\zeta_T)=0\right\}=\mathbb{P}\left\{L_T^\alpha(\zeta_T)>0\right\}-\varepsilon \geq {1\over 2} \mathbb{P}\left\{L_T^\alpha(\zeta_T)>0\right\}>0. 
\end{eqnarray}
Consider now a decreasing sequence $\varepsilon_n\to 0$. By construction above, there is an increasing sequence $N_n\to \infty$ such that 
$\mathbb{P}\left\{\forall_t \ N_n\xi_t\in -G_t\right\}\geq 1-T\varepsilon_n$, $\mathbb{P}\left\{L_T(N_n\zeta_T)\geq 0\right\}\geq 1-\varepsilon_n$ and $\mathbb{P}\left\{L_T(N_n\zeta_T) > 0\right\}\geq {1\over 2} \mathbb{P}\left\{L_T^\alpha(\zeta_T)>0\right\}>0$. Therefore we have a weak asymptotic arbitrage on $(G_t)$ market.  
\end{proof}
In what follows we shall call by $\bar{K}$ market the market with solvent set $\bar{K}_t$ at time $t$ and  by $K$ 
market the market with 
solvent set $K_t$ at time $t$, where $K_t$ is defined in \eqref{K set}. We have
\begin{Proposition}
Assuming (B1) asymptotic arbitrage on $\bar{K}$ market implies  asymptotic arbitrage on $K$ market. 
\end{Proposition}
 \begin{proof}
Assume we have a sequence $\zeta_T^n \in R_T^0(\bar{K}):=\sum_{i=0}^T L^0(-\bar{K}_i,{\cal{F}}_i)$
such that  $\lim_{n\to \infty}\mathbb{P}\left\{\bar{L}_T(\zeta_T^n)\geq 0\right\}=1$ and $\liminf_{n\to \infty} \mathbb{P}\left\{\bar{L}_T(\zeta_T^n) > 0\right\}>0$ and 
$\zeta_T^n=\sum_{i=0}^T \xi_i^n$, where $\xi_i^n=(\xi_i^{1,n},\xi_i^{2,n})$ is ${\cal F}_i$ measurable and takes values in $-\bar{K}_i$. 
Let $\tilde{\zeta}_T^n=\sum_{i=0}^T \tilde{\xi}_i^n$, where $\tilde{\xi}_i^n=({\xi}_i^{1,n},{\xi}_i^{2,n}-{1\over n})$ whenever ${\xi}_i^{1,n} {\xi}_i^{2,n}<0$ and $\tilde{\xi}_i^n=({\xi}_i^{1,n},{\xi}_i^{2,n})$ otherwise. Then $\tilde{\xi}_i^n\in 
-{K}_i$ and $|\bar{L}_T(\tilde{\zeta}_T^n)-\bar{L}_T(\zeta_T^n)|\leq {1\over n} (T+1)\|\overline{S}_T(\infty)\|_\infty$. 
Consequently 
$\liminf_{n\to \infty}\mathbb{P}\left\{\bar{L}_T(\tilde{\zeta}_T^n)\geq 0\right\}=1$ and  $\liminf_{n\to 
\infty}\mathbb{P}\left\{\bar{L}_T(\tilde{\zeta}_T^n) > 0\right\}>0$ from which an asymptotic arbitrage on $K$ market follows. 
  \end{proof}
We can now formulate our main result
\begin{Theorem}\label{thmm}
Under (B1) and (B2) asymptotic arbitrage on $\bar{K}$ market implies weak asymptotic arbitrage on $(G_t)$ market. 
\end{Theorem}
\begin{proof}
Assume that $\zeta_T^n=\sum_{i=0}^T \xi_i^n$, where $\xi_i=(\xi_i^{1,n},\xi_i^{2,n})\in  L^0(-\bar{K}_i,{\cal{F}}_i)$ forms 
an asymptotic arbitrage on $\bar{K}$ market i.e. $\liminf_{n\to \infty}\mathbb{P}\left\{\bar{L}_T(\zeta_T^n)\geq 0\right\}=1$ and $\liminf_{n\to \infty}\mathbb{P}\left\{\bar{L}_T(\zeta_T^n) > 0\right\}>0$. 
Define $\tilde{\xi}_{i}^{2,n}={\xi_i^{2,n} \over 1 - {1\over n}}$ for 
$\xi_i^{2,n}<0$ and $\xi_i^{1,n}\xi_i^{2,n}<0$, and $\tilde{\xi}_{i}^{2,n}={\xi_i^{2,n} \over 1 + {1\over n}}$ for $\xi_i^{2,n}>0$ and 
$\xi_i^{1,n}\xi_i^{2,n}<0$, and $\tilde{\xi}_{i}^{2,n}=\xi_{i}^{2,n}$ in the other cases. Let $\tilde{\xi}_i^{1,n}=\xi_i^{1,n}$ and $\tilde{\xi}_i^n=(\tilde{\xi}_i^{1,n},\tilde{\xi}_i^{2,n})$. 
Therefore $\tilde{\zeta}_T^n:=\sum_{i=0}^T \tilde{\xi}_i^n=\zeta_T^n - \sum_{i=0}^T 
(\xi_i^{1,n},\eta(n,\xi_i^{2,n})\xi_i^{2,n})$, where $\eta(n,\xi_i^{2,n})={-{1\over n}\over 1-{1\over n}}\xi_i^{2,n}$ for 
$\xi_i^{2,n}<0$ and $\xi_i^{1,n}\xi_i^{2,n}<0$, 
$\eta(n,\xi_i^{2,n})={{1\over n}\over 1+{1\over n}}\xi_i^{2,n}$ for  $\xi_i^{2,n}>0$ and $\xi_i^{1,n}\xi_i^{2,n}<0$ and 
$\eta(n,\xi_i^{2,n})=0$ in the other cases. Clearly $|\eta(n,\xi_i^{2,n})|\leq {{1\over n} \over 1-{1 \over n}}$. 
Furthermore $\tilde{\zeta}_T^n \in R_T^0(K^{1\over n})=\sum_{t=0}^T L^0(-{K}_t^{1\over n},{\cal{F}}_t)$ and under (B1) and 
(B2) $\liminf_{n \to \infty}\mathbb{P}\left\{\bar{L}_T^{1\over n}(\tilde{\zeta}_T^n)\geq 0\right\}=1$ and $\liminf_{n 
\to \infty}\mathbb{P}\left\{\bar{L}_T^{1\over n}(\zeta_T^{n}) > 0\right\}>0$. For given $n$ and $0<\varepsilon_n \to 0$ and each 
$t=0,1,\ldots, T$ there is a deterministic $d_t^n$ such that $\mathbb{P}\left\{|\tilde{\xi}_t^{2,n}|\geq d_t^n \ 
or \  \tilde{\xi}_t^{2,n}=0\right\}\geq 1-\varepsilon_n$ and therefore $\liminf_{n \to \infty}\mathbb{P}\left\{|\tilde{\xi}_{t}^{2,n}|\geq d_t^n \ or \ \tilde{\xi}_{t}^{2,n}=0\right\}\geq 1-\varepsilon_n$. As in the 
proof of Proposition \ref{prop3} for each $t$ there is a deterministic $m_{t,n}$ such that 
$m_{t,n}\tilde{\xi}_t^n \in -G_t^0$ whenever $|\tilde{\xi}_{t}^{2,n}|\geq d_t^n$. Consequently for 
$m(n)=m_{0,n}+\ldots m_{T,n}$ we have $ \mathbb{P}\left\{\forall_t \ m(n)\tilde{\xi}_t^n \in -G_t\right\}\geq 1-T\varepsilon_n$. Moreover there is deterministic 
$d_n>0$ such that $\mathbb{P}\left\{m(n)|\tilde{\zeta}_T^{2,n}|\geq d_n \ or 
\ \tilde{\zeta}_T^{2,n}=0 \right\}\geq 1-\varepsilon_n$, where $\tilde{\zeta}_T^n=:(\tilde{\zeta}_T^{1,n}, 
\tilde{\zeta}_T^{2,n})$. 
By Lemma \ref{Lem4} there is a deterministic $\bar{m}(n)$ such that  
$ \mathbb{P}\left\{\bar{m}(n)m(n)\tilde{\zeta}_T^n\in G_T \right\}\geq 
1-\varepsilon_n$ and $ \mathbb{P}\left\{L_T(\bar{m}(n)m(n)\tilde{\zeta}_T^n)>0 
\right\}>0$. Therefore following the arguments of the proof of Proposition \ref{prop3} we obtain that there is an increasing sequence $N_n\to \infty$ such that $ \mathbb{P}\left\{N_n\tilde{\zeta}_T^n \in R_T^0\right\}\geq 1-T\varepsilon_n$, $ \mathbb{P}\left\{L_T(N_n\tilde{\zeta}_T^n)\geq 0 \right\}\geq 1-\varepsilon_n$ and 
\begin{equation}
\mathbb{P}\left\{L_T(N_n\tilde{\zeta}_T^n)>0 \right\}\geq {1\over 2}\mathbb{P}\left\{L_T(N_n\tilde{\zeta}_T^n)>0 
\right\}.
\end{equation}
Since 
\begin{equation}
\liminf_{n\to \infty}\mathbb{P}\left\{L_T(N_n\tilde{\zeta}_T^n)>0\right\}= \liminf_{n\to \infty}\mathbb{P}\left\{\bar{L}_T^{1\over n}(N_n\tilde{\zeta}_T^n)>0\right\}>0
\end{equation}
we obtain an asymptotic weak arbitrage on $(G_t)$ market. 
\end{proof}

\section{Markets with non infinitely divisible assets}

In this section we consider markets with bid and ask prices $(\underline{S}_t(m))$ and $(\overline{S}_t(l))$ 
assuming that $l$, and $m$ can take nonnegative integer values only. We assume that they satisfy (a1)-(a5) and 
convexity, concavity are understood as increasing or decreasing slopes of the lines between positive integer 
points of the graph. The set of solvent positions $G_t^N$ at time $t$ is given by 
\begin{equation}
G_t^N:=\left\{(x,y)\in \mathbb{R}\times \mathbb{Z}: L_t(x,y)\geq 0\right\},
\end{equation} 
where $Z$ denotes integer numbers. Clearly $G_t^N=G_t\cap (\mathbb{R}\times \mathbb{Z})$.
We call such market $(G_t^N)$ market. Consequently we can 
define by analogy to \eqref{r01} and \eqref{r02}
\begin{equation}
R^0_{T,N}=\sum_{t=0}^T L^0(-G_{t}^N,{\cal F}_{t})
\end{equation}
and
\begin{equation}\label{rn02}
LV^0_{T,N}=\left\{L_T(V_T): V_T\in R^0_{T,N}\right\}.
\end{equation}
Then {\it absence of  arbitrage condition} ($NA^N$) on our market with non divisible assets is when $LV^0_{T,N} \cap 
L^0(\mathbb{R}_+,{\cal F}_T)=\left\{0\right\}$. Similarly we can define $K_t^N:=K_t \cap (\mathbb{R}\times 
\mathbb{Z})$ and $\bar{K}_t^N:=\bar{K}_t \cap (\mathbb{R}\times \mathbb{Z})$. Clearly $\bar{K}_t^N$ is the set 
of solvent positions on the market with proportional transaction costs with bid price $\underline{S}_t(\infty)$ 
and ask price $\overline{S}_t(\infty)$ when we sell or buy nonnegative integer number of assets. We call this market $(\bar{K}_t^N)$ market.
Let 
\begin{equation}
\bar{R}^0_{T,N}=\sum_{t=0}^T L^0(-\bar{K}_{t}^N,{\cal F}_{t})
\end{equation}
and
\begin{equation}\label{rnk02}
\widebar{LV}^0_{T,N}=\left\{L_T(V_T): V_T\in \bar{R}^0_{T,N}\right\}.
\end{equation}

We can also define absence of arbitrage version ($\widebar{NA}^N$) on $(\bar{K}_t^N)$ market as $\widebar{LV}_T^{0,N}\cap L^0(\mathbb{R}_+)=\left\{0\right\}$. Similarly let 
\begin{equation}
\tilde{R}^0_{T,N}=\sum_{t=0}^T L^0(-{K}_{t}^N,{\cal F}_{t})
\end{equation}
and for $(K_t^N)$ corresponding to solvent sets on $(K_t)$ market with non divisible assets we define
\begin{equation}\label{rnt02}
\widetilde{LV}^0_{T,N}=\left\{L_T(V_T): V_T\in \tilde{R}^0_{T,N}\right\}
\end{equation}
and we have absence of  arbitrage ($\widetilde{NA}^N$) when $\widetilde{LV}_T^{0,N}\cap 
L^0(\mathbb{R}_+)=\left\{0\right\}$.
The following version of Lemma \ref{equival} is satisfied 
\begin{Lemma}\label{equivalN}  We have
\begin{enumerate}
\item[(a)] under assumption $a_0=0$ the property ($NA^N$) is equivalent to $R^0_{T,N} \cap L^0((G_T^N)^0\cup 
    \left\{0\right\},{\cal F}_T)=\left\{0\right\}$; moreover if $a_0\neq 0$ we have only the implication: from $(NA^N$) it 
    follows that  $R^0_{T,N} \cap L^0((G_T^N)^0\cup
    \left\{0\right\},{\cal F}_T)=\left\{0\right\}$,
\item[(b)] $(\widebar{NA}^N)$ is equivalent to $\bar{R}^0_{T,N} \cap L^0(K_T^N,{\cal F}_T)=\left\{0\right\}$,
\item[(c)] $(\widetilde{NA}^N)$ is equivalent to $\tilde{R}^0_{T,N} \cap L^0(K_T^N,{\cal 
    F}_T)=\left\{0\right\}$.
\end{enumerate}
\end{Lemma}
Let now $M^{\alpha,N}$ be a market with proportional transaction costs with bid prices 
$(\underline{S}_t=(1-\alpha)\underline{S}_t(\infty))$ 
and ask prices $(\overline{S}_t=(1+\alpha)\overline{S}_t(\infty))$, where $0<\alpha \leq \min\left\{ 
{\overline{S}_t(0) \over \overline{S}_t(\infty)}-1, 1-{\underline{S}_t(0) \over 
\underline{S}_t(\infty)}\right\}$, $\mathbb{P}$ a.s. such that we can buy or sell only positive integer number 
of assets. 
We have
\begin{Proposition}
Under (B1) and (B2) arbitrage on the market $M^{\alpha,N}$ implies arbitrage on $(G_t^N)$ market (with non infinitely divisible assets). 
\end{Proposition}
\begin{proof}
Assume that there is $\zeta_T \in R^0_{T,N}(K^\alpha):=\sum_{t=0}^T L^0(-\bar{K}_t^{\alpha,N},{\cal{F}}_t)$, 
where $\bar{K}_t^{\alpha,N}:=\bar{K}_t^{\alpha}\cap (\mathbb{R}\times \mathbb{Z})$, such that 
$\mathbb{P}\left\{L_T^\alpha(\zeta_T)\geq 0\right\}=1$ and $\mathbb{P}\left\{L_T^\alpha(\zeta_T) > 0\right\}>0$ 
and 
$\zeta_T=\sum_{i=0}^T \xi_i$, where $\xi_i=(\xi_i^1,\xi_i^2)$ is ${\cal F}_i$ measurable and takes values in 
$-\bar{K}_i^{\alpha,N}$. By Lemma \ref{Lem4} for each  $t=0,1,\ldots,T$ there is a positive integer $n_t$ such 
that 
$n_t \xi_t\in -G_t^N$. Consequently for $n=n_0+n_1+\ldots + n_T$ we 
have $n \xi_t \in -G_t^N$. Therefore for $n\zeta_T$ we have an arbitrage on the market $M^{\alpha,N}$. By Lemma 
\ref{Lem4} there is a positive integer $\bar{n}$ such that  $\bar{n}n\zeta_T$ is in 
$(G_T^N)^0\cup\left\{0\right\}$ and forms an arbitrage on $(G_t^N)$ market. 
\end{proof}
By analogy to Theorem \ref{thmm} we have the following result
\begin{Theorem}
Under (B1) and (B2) arbitrage on $(\bar{K}_t^N)$ market implies
 asymptotic 
arbitrage on $(G_t^N)$ market. 
\end{Theorem}
\begin{proof}
We adapt the proof of Theorem \ref{thmm}. Assume that $\zeta_T=\sum_{i=0}^T \xi_i$, where 
$\xi_i=(\xi_i^1,\xi_i^2)\in  L^0(-\bar{K}_i,{\cal{F}}_i)$ and $\xi_i^2\in \mathbb{Z}$ forms an arbitrage on 
$\bar{K}$ market i.e. $\mathbb{P}\left\{\bar{L}_T(\zeta_T)\geq 0\right\}=1$ and 
$\mathbb{P}\left\{\bar{L}_T(\zeta_T) > 0\right\}>0$. Define $\xi_{i}^{1,n}=\xi_i^1 (1- {1\over n})$ for 
$\xi_i^2<0$ and $\xi_i^1\xi_i^2<0$ and $\xi_{i}^{1,n}=\xi_i^1(1 + {1\over n})$ for $\xi_i^2>0$ and 
$\xi_i^1\xi_i^2<0$ and $\xi_{i}^{1,n}=\xi_{i}^1$ in the other cases. Let $\xi_{i}^{2,n}=\xi_i^2$. 
Therefore $\zeta_T^n:=\sum_{i=0}^T \xi_{i}^n=\zeta_T - \sum_{i=0}^T 
(\eta(n,\xi_i^1)\xi_i^1,0)$, where $\eta(n,\xi_i^1)={1\over n}\xi_i^1$ for $\xi_i^2<0$ and 
$\xi_i^1\xi_i^2<0$, 
$\eta(n,\xi_i^1)=-{1\over n}\xi_i^1$ for  $\xi_i^2>0$ and $\xi_i^1\xi_i^2<0$ and $\eta(n,\xi_i^1)=0$ in 
the other cases. Then $\zeta_T^n \in R_T^0(\bar{K}^{1\over n})=\sum_{t=0}^T 
L^0(-\bar{K}_t^{1\over n},{\cal{F}}_t)$ and under (B1) and (B2) $\liminf_{n \to \infty
}\mathbb{P}\left\{L_T^{1\over n}(\zeta_T^n)\geq 0\right\}=1$ and $\liminf_{n \to \infty }\mathbb{P}\left\{L_T^{1\over n}(\zeta_T) > 0\right\}>0$. By Lemma \ref{Lem4} there is a positive integer 
$m_{t,n}$ such that $m_{t,n}\xi_{i}^n\in -G_t$ and for $m(m)=m_{0,n}+\ldots 
m_{T,n}$ we have $m(n)\zeta_T^n \in R_T^0$. Under (B1) and (B2) 
$\liminf_{n \to \infty}\mathbb{P}\left\{L_T^{1\over n}(\zeta_T^n)\geq 0\right\}=1$ and $\liminf_{n \to \infty 
}\mathbb{P}\left\{L_T^{1\over n}(\zeta_T^n) > 0\right\}>0$, and the same holds for $m(n)\zeta_T^n$.
By Lemma \ref{Lem4} again there is a positive integer $\bar{m}(n)$ such that 
$\bar{m}(n)m(n)\zeta_T^\alpha\in G_T^0$ and $\liminf_{n \to \infty }\mathbb{P}\left\{L_T^{1\over n} (\bar{m}(n)m(n)\zeta_T^n) > 0\right\}>0$, which completes the proof of asymptotic arbitrage.  
\end{proof}

\noindent{\bf Acknowledgements.} The authors are greateful to the reviewers for helpful comments and suggestions.

\end{document}